\documentclass[11pt]{amsart}

\usepackage{amssymb,amsmath,amsthm,amsfonts,amsopn,url,xcolor,hyperref,enumerate,mathtools,microtype,MnSymbol,comment,amsrefs,enumitem}
\usepackage[normalem]{ulem}
\usepackage{appendix}
\usepackage[all]{xy}
\input xy
\xyoption{all}
\usepackage{amscd}
\usepackage{soul}
\usepackage[mathscr]{euscript}

\usepackage{chngcntr}

\theoremstyle{plain}
\newtheorem{thm}{Theorem}[section]

\newtheorem{prop}[thm]{Proposition}
\newtheorem{lemma}[thm]{Lemma}
\newtheorem{cor}[thm]{Corollary}

\theoremstyle{definition}
\newtheorem{defn}[thm]{Definition}
\newtheorem*{defn*}{Definition}
\newtheorem*{question*}{Question}

\newtheorem{example}[thm]{Example}
\newtheorem*{example*}{Example}
\newtheorem{rem}[thm]{Remark}
\newtheorem*{rem*}{Remark}

\newtheorem{notation}[thm]{Notation}
\numberwithin{equation}{thm}

\counterwithin{figure}{subsection}

\newcommand{\field}[1]{\mathbb{#1}}

\newcommand{\Z}{\field{Z}}

\newcommand{\R}{\field{R}}
\newcommand{\F}{\field{F}}

\newcommand{\CC}{\mathcal{C}}

\newcommand{\func}[1]{\mathrm{#1} \,}
\newcommand{\rank}{\func{rank}}

\DeclareMathOperator{\pdim}{pd}

\newcommand{\be}{\begin{enumerate}}
\newcommand{\ee}{\end{enumerate}}

\newcommand{\cP}{\mathcal{P}}
\newcommand{\cX}{\mathcal{X}}
\newcommand{\cY}{\mathcal{Y}}
\renewcommand{\phi}{\varphi}

\newcommand{\fp}{\mathfrak{p}}
\newcommand{\fq}{\mathfrak{q}}
\newcommand{\fX}{\mathfrak{X}}
\newcommand{\fY}{\mathfrak{Y}}
\newcommand{\fN}{\mathfrak{N}}

\newcommand{\fg}{finitely generated}

\newcommand{\Tor}[4][i]{\operatorname{Tor}_{#1}^{#2}\!\left({#3},{#4}\right)}

\newcommand{\betti}[3][S]{\beta_{#2}^{#1}\!\left(#3\right)}

\newcommand{\reg}[1]{\operatorname{reg}\left({#1}\right)}
\newcommand{\pci}{polarized neural ideal }
\newcommand{\Ann}[2]{\operatorname{Ann}_{#1}\!\left({#2}\right)}
\newcommand{\Supp}[2][]{\operatorname{Supp}_{#1}\!\left(#2\right)}

\newcommand{\crd}{\color{teal}} 
\newcommand{\cfg}{\color{orange}}
\allowdisplaybreaks

\author{Hugh Geller}
\address{Center for Naval Analyses, Arlington, Virginia 22201 U.S.A.}
\email{\href{mailto:geller.hugh@gmail.com}{geller.hugh@gmail.com}}

\author{Rebecca R.G.}
\address{Department of Mathematical Sciences \\ George Mason University \\ Fairfax, VA  22030}
\email{rrebhuhn@gmu.edu}

\author{Alexandra Seceleanu}
\address{Department of Mathematics, University of Nebraska-Lincoln, 203 Avery Hall, Lincoln, NE 68588}
\email{aseceleanu@unl.edu}

\author{Nora Youngs}
\address{Department of Mathematics, Colby College, Waterville, ME 04910 }
\email{nora.youngs@colby.edu }

\title{Betti numbers of inductively pierced codes}

\date{\today}
\begin{document}

\begin{abstract}
    Neural ideals were introduced by Curto, Itskov, et al as an algebraic tool to study neural codes. In this paper, we use the notion of polarization introduced by G\"{u}nt\"{u}rk\"{u}n, Jeffries, and Sun to compute Betti numbers of inductively pierced neural codes. We prove that quadratically generated neural ideals are inductively pierced if and only if they have regularity 2. Further, we demonstrate how the Betti numbers yield information on the number and type of piercings of the original code.
    This work shows the utility of algebraic invariants of the neural ideal in detecting geometric features of the associated receptive fields.
\end{abstract}

\maketitle
\setcounter{tocdepth}{1} 

\section{Introduction}

In this paper we compute algebraic invariants of a certain family of neural ideals and demonstrate that they correspond to geometric features of the original neural code.

Neural rings were first introduced in \cite{NRingBMB} as an algebraic tool to study \textit{receptive field codes} (RF codes). RF codes arise in the brain in many contexts. One particular example is place cells - certain cells in the hippocampus that fire consistently every time an animal is in the same location within an environment (the stimulus space)
\cite{okeefe}. The set of locations where a particular cell fires is called its place field - or more generally, its receptive field. A neural code consists of a binary codeword for each distinct intersection of receptive fields that appears in the stimulus space (see Figure~\ref{fig:Ex1} for an example).
Typically, we consider stimulus spaces as subsets of a subspace $X$ of $\mathbb{R}^d$; in the case of place cells, $d=2$ or $3$, but there are other types of neurons whose receptive fields naturally lie in different dimensions (see \cite{NRingBMB} and references therein). 

The neural ring (and the related neural ideal) provide a tool for applying algebraic methods to these receptive field codes. Roughly speaking, the neural ideal is the set of polynomials that vanish at every codeword of a given neural code, excepting those that vanish on every possible codeword. These polynomials can then be interpreted to give descriptions of relationships between the receptive fields (such as containment or disjointness). This information, in turn, can be used to answer questions about what sort of arrangements of receptive fields could have generated the code. 

Some such questions (e.g. could a particular code have arisen from an arrangement of convex, open regions?) are quite difficult to answer for neural codes in general \cite{NCConvex}. However, there are some types of codes that are quickly detectable and/or whose neural ideals have particularly nice generators.

One such collection is {\it inductively pierced} codes \cite{GrossObatakeYoungs}. As the name suggests, these are codes that can be constructed iteratively, by adding receptive fields to an existing arrangement following strict conditions on how each receptive field may interact with those already present.  These codes are easily detectable from their algebraic signature. They have a strong connection to chordal graphs, and bounds on the best dimension in which an arrangement using convex regions exists are known (and if the receptive fields are spheres, the dimension can be exactly determined)\cite{curry2022recognizing}.

One obstruction to computing algebraic invariants of neural ideals is that they are generated by \textit{pseudomonomials}, products $\Pi_{i \in \sigma} x_i \Pi_{i \in \tau} (1-x_i)$, which are neither graded nor local and as a result are not as well studied in commutative algebra. To deal with this issue, G\"{u}nt\"{u}rk\"{u}n, Jeffries, and Sun give a process for \textit{polarizing} a neural ideal, turning it into a squarefree monomial ideal in the extension ring $\F_2[x_1,\ldots,x_n,y_1,\ldots,y_n]$ \cite{neuralpolarization}. One can then compute invariants such as Betti numbers for this ideal instead, and they carry over to information on the original neural ideal.

The primary goal of this paper is to compute homological invariants of polarized neural ideals and describe how
they correspond to geometric features of  the associated arrangement of receptive fields. In particular, we will focus on the Betti numbers for the polarized neural ideals of inductively pierced codes. Broadly, we find that this offers an algebraic invariant that provides geometric features of the code and information about its piercing structure, while remaining efficiently computable. More specifically, homological notions such as regularity and Betti numbers can be used to determine whether a neural code is inductively pierced, and if it is, how many piercings of each order it has.  If we use a multigrading that tracks degrees in $x$'s and $y$'s separately, the multigraded Betti numbers can further distinguish how many firing fields each piercing is contained in.
This demonstrates that homological invariants have the potential to be extremely useful in classifying neural codes.

The Castelnuovo–Mumford regularity of a homogeneous ideal measures its homological complexity. In particular, among nondegenerate ideals (those containing no linear forms), ideals of regularity 2 are precisely those whose higher-order relations are generated by linear syzygies. Our main results show a neural ideal is quadratic with regularity 2 if and only if the corresponding neural code is inductively pierced (Theorem~\ref{thm:reg2indpierced}). We also compute the Betti numbers (Theorem~\ref{thm:xybettis} and Corollary~\ref{cor:bettinumbers}) and projective dimension (Corollary~\ref{cor:pdim}) of inductively pierced codes. Further, we prove the Betti numbers of an inductively pierced code uniquely determine how many $k$-piercings the code has for each $k$, implying that the numbers of $k$-piercings are an invariant of the code (Corollary~\ref{cor:pierceToBetti} and Theorem~\ref{thm:jkl}). Consequently, we prove that the elimination ordering on a chordal graph has a fixed number of each type of simplicial vertex (Theorem~\ref{thm:simplicialfixed}).

Since we expect readers of this paper to have varied mathematical backgrounds,
we have included appendices on some of the algebraic and combinatorial background material.

\section{Definitions}\label{sec:defn}

\subsection{Neural codes/rings}

Here, we introduce the basic algebraic objects of interest: neural codes and neural rings. Unless otherwise noted, the definitions and results in this section are from \cite{NRingBMB}.

We will generally assume we have a set of neurons $i$, each of which fires in a particular open set $U_i$ in some affine space $\R^d$.

\begin{notation}
    We represent the set of neurons as $[n] = \{1,...,n\}$; the power set of $[n]$ is denoted $2^{[n]}$.
    
    $\F_2$ will be the field on two elements.
\end{notation} 

\begin{defn} A {\it neural code} $\CC$ is any collection of subsets of the set $[n]$ of all neurons: $$\CC \subseteq 2^{[n]}.$$   Each element $\sigma$ of the code is called a {\it codeword} and may be supposed to represent a data point where the neurons in $\sigma$ were simultaneously active.  As a convention, the element $\emptyset$ is always included in $\CC$.
\end{defn}

\begin{defn}
 A collection of regions $\mathcal U 
 = \{U_1,...,U_n\}$ where $U_i\subset \R^d$ is a {\it realization} of the code $\CC\subset 2^{[n]}$ if $$\CC = \{\sigma \subseteq [n] \,|\, \bigcap_{i\in \sigma} U_i \backslash \bigcup_{j\notin \sigma} U_j \neq \emptyset\}.$$ 
\end{defn}

\begin{example}\label{ex:key} Consider the code $\CC = \{\emptyset, 1,2,3,4,12,14,23,24,35,124,235\}$ Note that here we use the shorthand of writing, e.g., 123 instead of $\{1,2,3\}$. Figure~\ref{fig:Ex1} shows one possible realization of this code. 

\begin{figure}[h]
\includegraphics[width=2.5in]{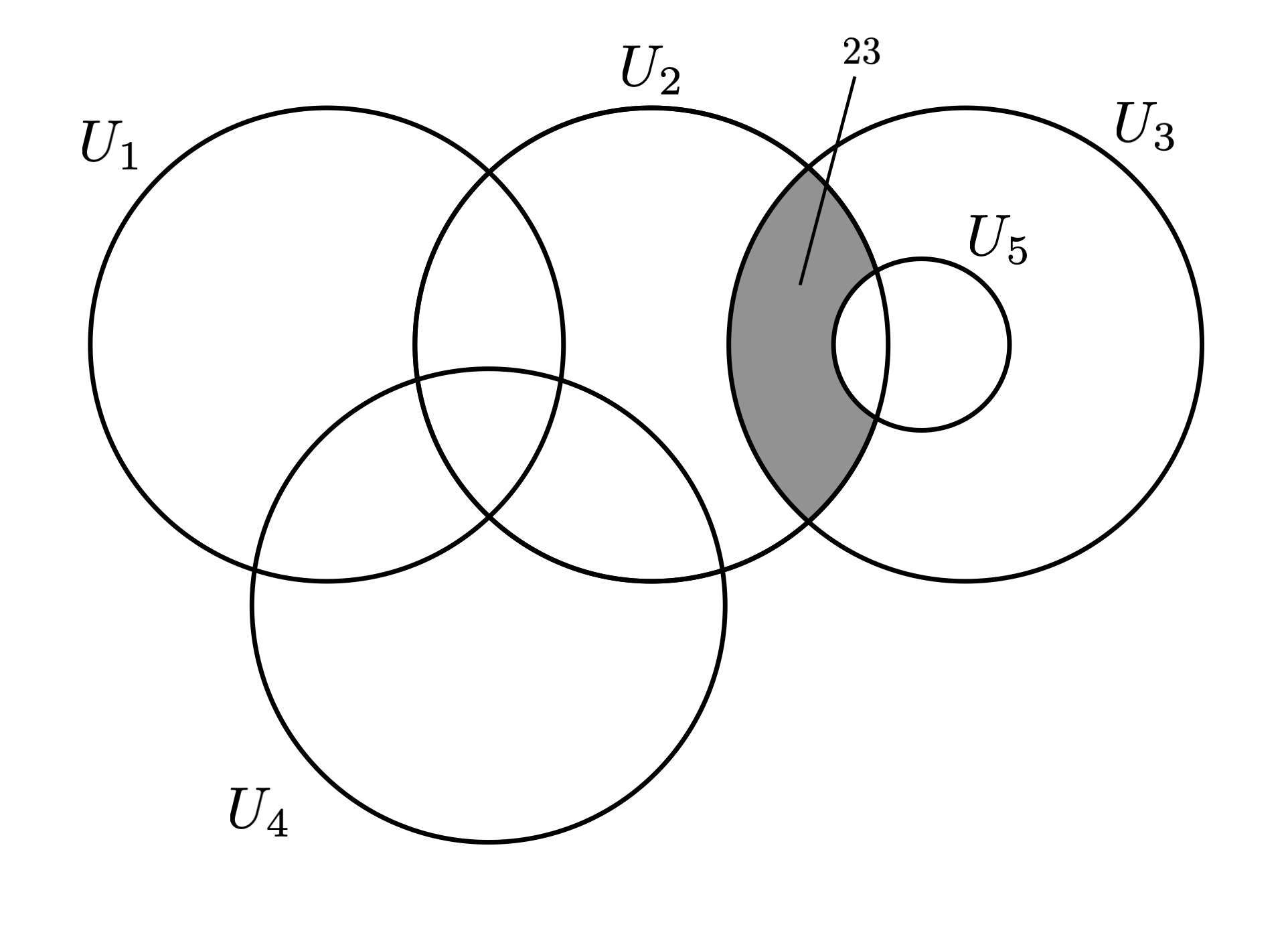}
\caption{A realization of the code $\CC$ from Example~\ref{ex:key}. The region corresponding to codeword $23$ is highlighted.}
\label{fig:Ex1}
\end{figure}
\end{example}

\begin{defn}\label{def:neuralring}
Given a neural code on $n$ neurons, we consider the associated ideal  $$I(\CC) = \{f \in R  \,|\, f(c) = 0\, \forall\, c\in \CC\}$$
of $R=\F_2[x_1,\ldots,x_n]$,
where $f$ is evaluated at $c$ by inputting a 1 for $x_i$ if $i \in c$ and a 0 for $x_i$ otherwise.

The {\it neural ring} is defined to be $$R_\CC = R/I(\CC).$$
\end{defn} 

The ideal $I(\CC)$ can be decomposed into two parts: $$I(\CC) = B+L$$ where $$B = \left( x_1(1-x_1),...,x_n(1-x_n) \right),$$ and $$L =\left ( \rho_v \,|\, v\in 2^{[n]}\backslash \CC \right )$$ where $\rho_v = \prod_{i\in v} x_i \prod_{j\notin v} (1-x_j)$. Note that $B$ is always the same, and reflects the use of the finite field $\F_2$, whereas the ideal $L$ is specific to the particular code $\CC$.  Thus, we will focus mostly on the properties of $L$, which we title the \textit{neural ideal}.\footnote{ In much of the literature, the notation for the neural ideal is $J$ instead of $L$. However, as we will rarely consider the neural ideal $L$ itself in this paper, we reserve $J$ for other uses (namely, its polarized version).} 

\begin{defn} \label{def:pseudomonomial}
A {\em pseudo-monomial} is a polynomial $f\in R$ of the form $f=\prod_{i\in \sigma} x_i \prod_{j\in \tau} (1-x_j)$, where $\sigma\cap\tau=\emptyset$. 
\end{defn} 

Observe that $L$ is generated by pseudo-monomials. In fact, these generators are, in a sense, maximal -- if we multiply them by any linear factor $x_i$ or $1-x_i$ not already part of the product, they would also be part of $B$, and thus redundant. The {\it minimal} pseudo-monomials of $L$ are more interesting, and provide another useful generating set for $L$.

\begin{defn}\label{def:cf} The {\em canonical form} of a code $\CC$, denoted $CF(\CC)$, is the set of all minimal pseudo-monomials of $L$, where a pseudo-monomial $f$ is said to be {\em minimal} if no other pseudo-monomial $g$ of $L$ divides $f$. If $L$ is the neural ideal of $\CC$, we may also refer to the canonical form of $\CC$ as $CF(L)$ without explicitly listing $\CC$.
\end{defn}

The canonical form generates the ideal $L$ (that is, $ L = ( CF(\CC) )$) and is often a friendlier collection of generators than the set $\rho_v$.

\begin{example}
Observe that the code $\CC$ from Example~\ref{ex:key} is a code on five neurons, with twelve codewords. If we use the generators $\rho_v$ for $L$ as in Definition~\ref{def:neuralring} then $L$ would have $2^5-12 = 20$ generators $\rho_v$, each of degree 5. For example, as $v=145$ is not in $\CC$, then $\rho_{145} = x_1x_4x_5(1-x_2)(1-x_3)$ would be one such generator.  

The canonical form (Definition~\ref{def:cf}), however, is much simpler. In this case, it contains only five pseudo-monomials, all of degree two: $$CF(\CC) = \{x_1x_3, \,x_1x_5, \, x_3x_4, \, x_4x_5, \, x_5(1-x_3)\}.$$ 
\end{example}

\begin{rem}
    Because the canonical form by definition contains {\it all} minimal pseudo-monomials, it may not be the {\it shortest} possible list of generators. For example, the code $\CC = \{\emptyset, 1, 12, 123\}$ has canonical form \[CF(\CC) = \{x_3(1-x_2), x_2(1-x_1), x_3(1-x_1)\}.\] But the ideal $L$ can also be generated by the shorter list $\{x_3(1-x_2), x_2(1-x_1)\}$, since $x_3(1-x_1)=(1-x_1)[x_3(1-x_2)]+x_3[x_2(1-x_1)]$.
\end{rem}

Neural ideals were originally studied using these pseudo-monomials; however, since they are neither standard graded nor all contained in any maximal ideal, many algebraic invariants do not apply. Instead, we convert neural ideals to monomial ideals using an operation called polarization. Monomial ideals are graded under the standard grading, or many multi-gradings, so we can compute homological invariants like Betti numbers for them.

\begin{notation}
Throughout this paper, except where otherwise specified, $R :=  \F_2[x_1,\ldots,x_n]$ will be the polynomial ring on $n$ variables, and \[S := \F_2[x_1,\ldots,x_n,y_1,\ldots, y_n]\] will be the polynomial ring on $2n$ variables. By default, both rings will be standard graded, except where a multi-grading is explicitly given. Neural ideals on $n$ neurons will live in $R$ and polarized neural ideals on $n$ neurons (which are graded) will live in $S$. All Betti numbers (see Appendix~\ref{sec:betti} for definitions) will be computed over $S$ unless otherwise noted. 
\end{notation}

\begin{defn}[{\cite{neuralpolarization}*{Section 3}}]
    We define a \textit{polarization} function $\cP$ (not a homomorphism) that sends pseudo-monomials in $R$ to squarefree monomials in $S$ via \[\prod_{i \in \sigma} x_i \prod_{i \in \tau} (1-x_i) \mapsto \prod_{i \in \sigma} x_i \prod_{i \in \tau} y_i.\]
    This does not extend to a ring homomorphism.

    If $I$ is an ideal of $R$, then  $\cP(I)$ is the ideal of $S$ generated by $\cP(f)$ for all pseudo-monomials $f \in I$.

    We also define a \textit{depolarization} map $d:S \to R$, which is the ring homomorphism sending $x_i \mapsto x_i, y_i \mapsto 1-x_i$. The map $d$ induces an isomorphism of $S$-modules $S/\left(\sum_{i=1}^n (x_i+y_i-1)\right) \cong R$.
\end{defn}

\begin{thm}[{\cite{neuralpolarization}*{Theorem 3.2}}]
\label{thm:canonpolarizes}
Let $\{g_1,\ldots,g_k\} \subseteq R$ be the canonical form of a neural ideal $L$. Then $\cP(L)=(\cP(g_1),\ldots,\cP(g_k))$.
\end{thm}

\begin{defn}
    The \textit{polarized canonical form} of a neural ideal is the list $\{\cP(g_1),\ldots,\cP(g_k)\} \subseteq S$, where $\{g_1,\ldots,g_k\} \subseteq R$ is the canonical form of a neural ideal $L$. We will often refer to the polarized canonical form of a neural ideal without specifying $L$, but $L$ can be found by depolarizing the elements of the polarized canonical form.

    We will refer to the ideal $J$ of $S$ generated by the polarized canonical form of a neural ideal as the \textit{polarized neural ideal}. This is a squarefree monomial ideal, so it is graded.
\end{defn}

One of the primary goals of this paper is to classify neural ideals by their canonical resolutions.

\begin{defn}[{\cite{neuralpolarization}*{Definition 4.5}}]
    Let $L \subseteq R$ be a neural ideal with canonical form $CF(L)$, and $J \subseteq S$ its \pci. The \textit{canonical resolution} of $L$ is $G_\bullet=F_\bullet \otimes_S S/(x_1+y_1-1,\ldots,x_n+y_n-1)$, where $F_\bullet$ is a minimal free resolution of $J$ over $S$ (which exists because $S/J$ is graded).
\end{defn}

\begin{rem}
\label{rem:canonres}
    Since this construction involves taking a quotient by a regular sequence, $G_\bullet$ is a resolution of $R/L$ by Corollary~\ref{cor:freeresregseq}.
    Further, the canonical resolution $G_\bullet$ of $L$ has the same numbers of copies of $R$ in each homological degree as $F_\bullet$ has copies of $S$, and the maps $\partial_\bullet^G$ are given by replacing each $y_i$ from the corresponding map $\partial_\bullet^F$ by $1-x_i$, also by Corollary~\ref{cor:freeresregseq}.

    As a result, the Betti numbers of $J$ give us the shape of the canonical resolution of $L$. We will primarily discuss the resolution of $J$ over $S$, so as to be able to use terms like ``Betti numbers'', but we will still gain information on the canonical resolution.
\end{rem}

Consequently, we make the following definition:

\begin{defn}
    Given Remark~\ref{rem:canonres}, we define the \textit{Betti numbers of a neural code $\CC$} by setting $\beta_i(\CC)$ to be the number of copies of $R$ in the $i$th homological degree of the canonical resolution of $\CC$.
\end{defn}

\begin{lemma}
\label{lem:polarizedmingens}
    Let $J=(g_1,\ldots,g_k)$ be a \pci, i.e. \[\cP(CF(L))=\{g_1,\ldots,g_k\}\] for some neural ideal $L \subseteq R$. 
    Then $\{g_1,\ldots,g_k\}$ is a minimal generating set for $J$.
\end{lemma}

\begin{proof}
By Proposition 1.3.5 of \cite{monomialideals}, this is a minimal generating set if and only if it is irredundant, i.e. no $g_i$ is divisible by any $g_j$ for $j \ne i$. But this is true as no element of $CF(L)$ divides any other, and by \cite{neuralpolarization}*{Lemma 3.1} polarization preserves this property. 
\end{proof}

\subsection{Inductively Pierced Codes}

The literature on neural codes is considerable and addresses many questions - which codes are convex? In what dimension? How can these codes be detected? (See \cites{NCConvex,NRingBMB} and related literature.) These questions typically do not have simple answers. However, many classes of codes have been identified that do have nice properties, some of which are easy to detect. One such class of codes is those that are {\it inductively pierced}.  Informally, a code is inductively pierced if it can be constructed iteratively, by adding new neurons to the code to interact with existing neurons only in restricted and predictable ways. To formally define these codes, we need to build up a bit of notation. The following definitions are from \cite{curry2022recognizing}, but note that we have rephrased any algebraic results to use the polarized form of the neural ideal.

\begin{defn} For any $\sigma\subseteq \tau\subseteq [n]$, the \emph{interval} between $\sigma$ and $\tau$ is the set 
\[
[\sigma,\tau] = \{\gamma\subseteq [n]\mid \sigma\subseteq \gamma\subseteq \tau\}.
\]
The \emph{rank} of an interval $[\sigma,\tau]$ is equal to $|\tau|- |\sigma|$. 
\end{defn}

\begin{defn} Given a code $\CC\subseteq 2^{[n]}$ and a neuron $i\in[n]$, we can construct the code $\CC\backslash i$, called the \emph{deletion} of $i$, 
which is obtained by deleting $i$ from every codeword in $\CC$ where it appears. For any $\sigma\subseteq [n]$, we let $\CC\backslash \sigma$ denote the code obtained by deleting all neurons $i\in\sigma$ from $\CC$. 
\end{defn}

\begin{defn}
\label{def:kpiercing}
Let $\CC\subseteq 2^{[n]}$ be a code and let $i\in[n]$ be a neuron.
We say that $i$ is an \emph{(abstract) $k$-piercing} of $\CC$ if there exist $\sigma\subseteq \tau\subseteq [n]\backslash\{i\}$ so that $[\sigma,\tau]$ has rank $k$ and the following two conditions are satisfied:
\begin{itemize}
    \item[(i)]$[\sigma,\tau]$ is contained in $\CC\backslash i$, and 
    \item[(ii)] $\CC = (\CC\backslash i)\cup [\sigma\cup\{i\}, \tau\cup\{i\}]$.
\end{itemize}
\end{defn}

More informally, to have a $k$-piercing, every element of this interval of rank $k$ must appear both with and without $i$, and these must be the only appearances of $i$. If we think of this from the perspective of a realization, the region $U_i$ is contained within $U_\sigma$, all of the neurons in $\tau\backslash\sigma$ appear in all combinations within $U_\sigma$, and all of the neurons in $(\tau\backslash\sigma) \cup \{i\}$ appear in all combinations within $U_\sigma$.

The above definition tells us how to identify if a given neuron is a $k$-piercing of the others. More often, however, we are interested in building up a code by adding neurons in succession so that each is a $k$-piercing of the ones that precede it. Codes that can be built this way are called inductively pierced.

\begin{defn}  A code $\CC\subseteq 2^{[n]}$ is {\it $k$-inductively pierced} if $\CC=\{\emptyset\}$, or there exists a neuron $i\in[n]$ which is a $k'$-piercing of $\CC$ for some $k'\le k$, and $\CC\backslash i$ is $k$-inductively pierced.

By the iterative nature of this definition, any inductively pierced code must have an ordering $i_1<i_2<\cdots< i_n$ of the neurons $[n]$ so that for each $j$, the neuron $i_j$ is a $k'$-piercing of the code $\CC\backslash\{i_{j+1},\ldots,i_n\}$ for some $k'\le k$.
We call such an order a {\it piercing order}.
\end{defn}
 
In general, an inductively pierced code may have many valid piercing orders, and so we do not assume this ordering is unique.   Example~\ref{ex:prcgorder} provides a code with multiple possible piercing orders. These orders do share some properties, which we list here.

\begin{enumerate}
\item  For any $k'$, a $k'$-piercing can only be constructed on a $k'-1$-piercing. As a result $i_j$ is at most a $j-1$-piercing, and in particular $i_1$ is a 0-piercing.

\item Any two piercing orders for the same code contain the same number of $k'$ piercings, for any $k'\leq k$. We will prove this result later (see Corollary~\ref{cor:pierceToBetti}).
\end{enumerate}

\begin{figure}
\includegraphics[width=4in]{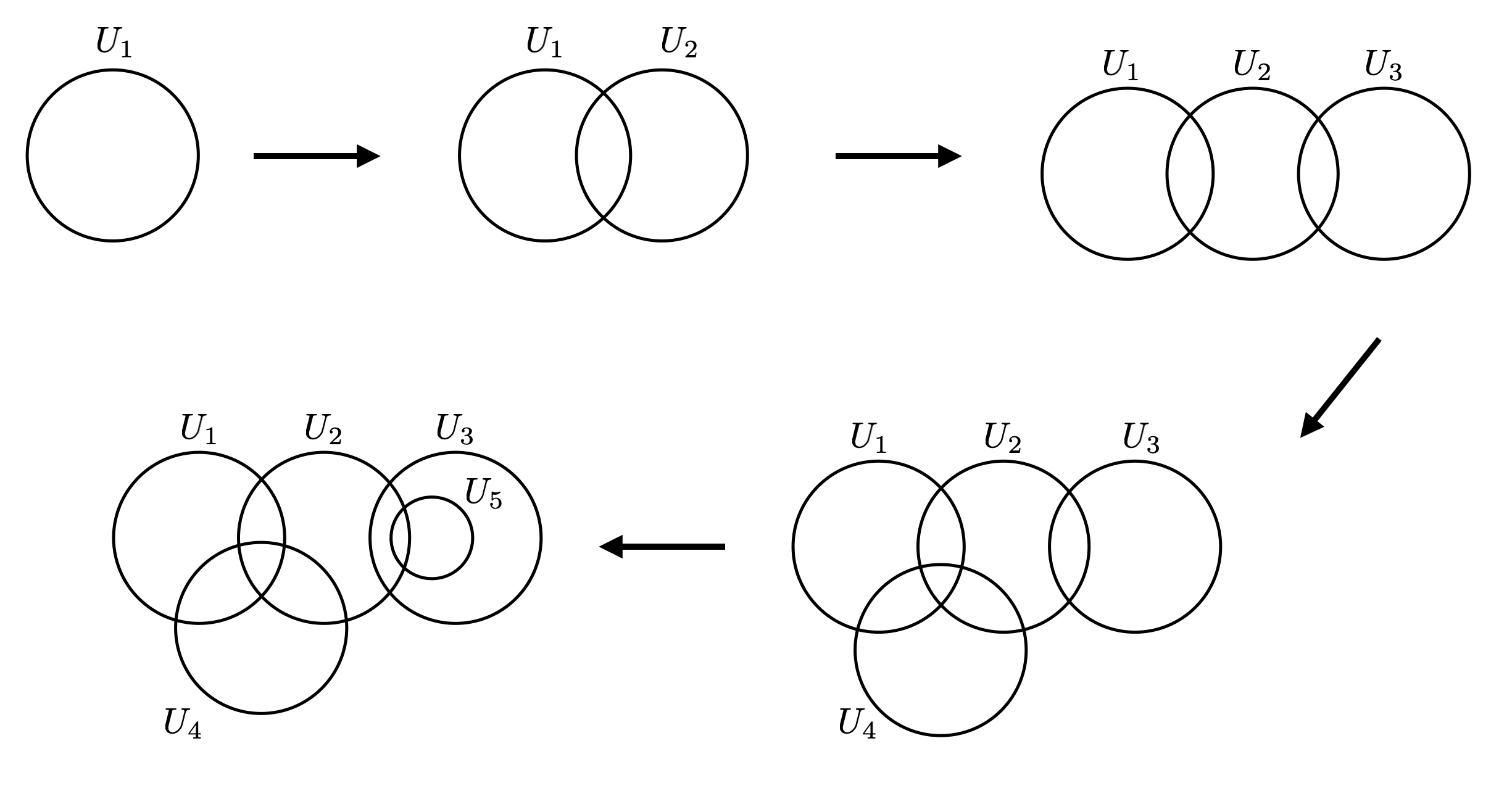}
\caption{Building the code from Example~\ref{ex:prcgorder} from the piercing order 1,2,3,4,5.
}
\label{fig:12345}
\end{figure}

\begin{example}\label{ex:prcgorder}
The code $\CC$ from Example~\ref{ex:key} is inductively pierced. We can build the code using  a succession of piercings as follows (see Figure~\ref{fig:12345}):

First, start with $\CC_0 = \{\emptyset\}$; this is inductively pierced by definition.

We then can take code $\CC_1 = \{\emptyset, 1\}$, where neuron 1 is a $0$-piercing (in this case, $\sigma = \tau = \emptyset$). Note that $\CC_1\backslash 1 = \CC_0$.

The code $\CC_2 = \{\emptyset, 1, 2, 12\}$  has neuron 2 as a 1-piercing (where $\sigma = \emptyset$, $\tau = 1$). Here, $\CC_2\backslash 2 = \CC_1$.

The code $\CC_3 = \{\emptyset, 1, 2, 3, 12, 23\}$ has neuron 3 as another 1-piercing (with $\sigma = \emptyset$, $\tau = 2$). Here, $\CC_3 \backslash 3 = \CC_2$.

The code $\CC_4 = \{\emptyset, 1, 2, 3, 4, 12, 23, 14, 24, 124 \}$ contains neuron $4$ as a 2-piercing of $1$ and $2$ (so $\sigma = \emptyset, \tau = 12)$. Here, $\CC_4 \backslash 4 = \CC_3$.

And finally, the code $\CC_5 = \CC = \{\emptyset, 1, 2, 3, 4, 12, 23, 14, 24, 35, 124, 235 \}$  has $5$ as a 1-piercing of 2 (within $3$), so $\sigma = 3$ and $\tau = 23$. Here, $\CC\backslash 5 = \CC_4$. 

The piercing order in this case is the same as the labels: 1,2,3,4,5. However, the code can also be built using the order $1,4,2,3,5$ as depicted in Figure~\ref{fig:14235} --note that in this case, neuron $4$ would be a 1-piercing with $\tau=1$, but $2$ would now be a $2-$piercing with $\tau=12$. 

However, not all orders are valid piercing orders - for example, $5,3,1,2,4$ is not valid as $3$ is not a piercing (of any order) of the code consisting of only neuron $5$.
\end{example}

\begin{figure}
\includegraphics[width=4in]{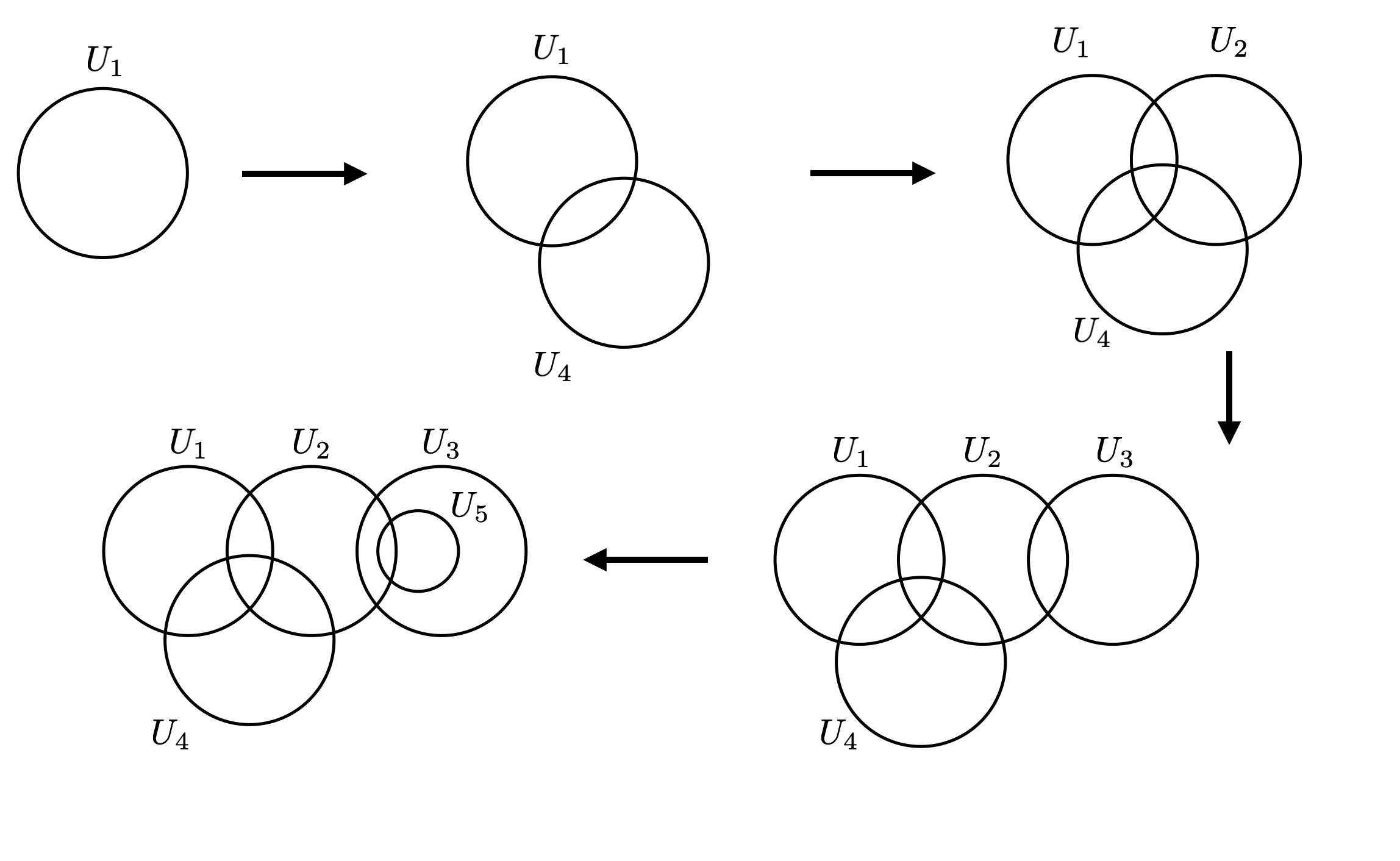}
\caption{Building the code from Example~\ref{ex:prcgorder} in the label order 1,4,2,3,5.}
 \label{fig:14235}
\end{figure}

Inductively pierced codes are detectable using the canonical form and an associated graph. This result (and the related definitions) are from \cite{curry2022recognizing}, but we provide it in its polarized form, since that is the form in which  we will use it. We note that polarization preserves the degrees of the generators of the canonical form, and the \pci of a neural code $\CC$ contains a binomial $x_ix_j$ (resp. $x_iy_j$,  $x_jy_i$,  $y_iy_j$) if and only if the canonical form of $\CC$ contains $x_ix_j$ (resp. $x_i(1-x_j)$,  $x_j(1-x_i)$,  $(1-x_i)(1-x_j)$), so polarization preserves the properties relevant to these results.
For this theorem, we need two additional definitions:

\begin{defn}\label{def:grc} Let $\CC$ be a code on $n$ neurons that has a degree two canonical form. The \textit{general relationship graph} $G(\CC)$ is the graph with vertices $V = [n]$, and an edge ${i,j}$ if and only if the \pci $J$ of $\CC$ does {\it not} contain any of the pseudo-monomials $x_ix_j$, $x_iy_j$, $x_jy_i$. 

[Note: because of the convention $\emptyset\in \CC$, the monomial $y_iy_j$ never appears in any \pci.]
\end{defn}

\begin{defn} A graph $G$ is {\it chordal} if it has no chordless cycle (every cycle of length $\geq 4$ has a chord) \cite{West}*{Definition 5.3.12}.

One notable property of chordal graphs is that they can be built up iteratively by adding so-called simplicial vertices; the similarity of inductive construction is not accidental.
\end{defn}

\begin{thm}[{\cite{curry2022recognizing}*{Theorem 6}}]
\label{thm:IpChar}
A neural code $\CC$ is inductively pierced if and only if
\begin{enumerate}
\item $CF(\CC)$ consists of degree two polynomials  and 
\item the graph $G(\CC)$ is chordal.
\end{enumerate}
\end{thm}

Since determining whether a graph is chordal is relatively efficient, this result allows us to realistically detect whether a code is inductively pierced. Other results in \cite{curry2022recognizing} prove that such codes are always realizable using convex sets, and provide upper bounds on the 
 minimum dimension required to build such a realization. 

\begin{rem} The result as stated is only for strictly degree-two canonical forms. However, for practical purposes, degree-one terms can usually be dealt with via a minor modification of the code.  The term $y_i$ is impossible due to the requirement that $\CC$ contain the empty set. The term $x_i$ would indicate that neuron $i$ never appears in any set in $\CC$; in this case, we can replace $\CC$ with $\CC\backslash i$ and eliminate the problem altogether, without changing any important property of the original code.

Technically, the original result also requires that the canonical form contains at most one degree-two term for each pair $i,j$. However, due to properties of the canonical form, most combinations of such terms cannot possibly appear in the canonical form anyway.  The one combination which can occur ($x_iy_j$ and $x_jy_i$)  would indicate that neurons $i$ and $j$ are identical in $\CC$, and so we can work with an equivalent code where we replace both with one single neuron. In this paper, we typically assume that this situation does not occur.
\end{rem}

\section{Homological Invariants of Inductively Pierced Codes}
\label{sec:fiberproduct}

In this section, we leverage Theorem \ref{thm:reg2indpierced} to compute common homological invariants of $S$ modules, namely, Betti numbers and regularity. For background, see Appendix~\ref{sec:betti}. We include in this section some background on multi-gradings and their associated homological invariants.

\begin{notation}
\label{not:piercingideal}
    Throughout this section, $\CC$ will be a neural code on $n$ neurons, $n$ will be a $k$-piercing of $\CC \backslash \{n\}$ (the code obtained by deleting $n$ from $\CC$), and $J_{n}$ will be the \pci of $\CC$ in $S=\F_2[x_1,\ldots,x_{n},y_1,\ldots,y_{n}]$. We let $J_{n - 1}$ denote the \pci of $\CC \backslash \{n\}$, expanded to $S$ since this does not impact the homological invariants studied here. We will assume all other ideals discussed  live in $S$ unless otherwise specified.
\end{notation}

First we give a homological characterization of inductively pierced codes, which can be viewed as adding a condition to Theorem \ref{thm:IpChar}.

\begin{thm}
\label{thm:reg2indpierced}
    Let $\CC$ be a neural code on $n$ neurons and $J$ its \pci contained in $S=\F_2[x_1,\ldots,x_n,y_1,\ldots,y_n]$. Assume that:
    \begin{enumerate}
        \item the canonical form of $J$ is generated by quadratic forms,
        \item no union of $U_i$ covers all of $X$, and
        \item no $U_i=U_j$ for $i \ne j$.
    \end{enumerate}   
  Then  $\CC$ is inductively pierced if and only if $\reg{J}=2$.
\end{thm}

\begin{rem}\label{rem:onlyoneperpair}
    Notice that the listed conditions (1)-(3) imply that for each pair $i,j$, $J$ has at most one generator chosen from $x_ix_j$, $x_iy_j$, and $x_jy_i$. Further, these are the only generators of $J$.
\end{rem}

To prove this result, we need the following lemma. See Definition \ref{def:regSeq} for the definition of a regular sequence.

\begin{lemma}\label{lem:indRegSeq}
    Let $\CC$ be a neural code on $n$ neurons and $J$ its \pci contained in $S=\F_2[x_1,\ldots,x_n,y_1,\ldots,y_n]$. Assume the hypotheses (1)-(3) of Theorem \ref{thm:reg2indpierced}. Then the sequence $(x_1 - y_1, \ldots, x_n - y_n)$ is $S/J$-regular.
\end{lemma}
\begin{proof}
    Since $J$ is square-free and a neural ideal, no generator of $J$ is divisible by $x_i y_i$ for all $1 \leq i \leq n$ and we obtain the ideals $\fN_i$, $\fX_i$, and $\fX_i$ satisfying
    \[J = \fN_i + \fX_i x_i + \fY_i y_i\] where $x_i$ and $y_i$ do not divide any generator of $\fN_i$, $\fX_i$, and $\fY_i$. By \cite{GellerRG}*{Theorem~5.5}, we have $\fX_i \cap \fY_i \subseteq J$ for all $1 \leq i \leq n$. We also find that the images $\fX_i$ and $\fY_i$ in $S/J$, denoted $\overline{\fX_i}$ and $\overline{\fY_i}$ respectively, satisfy \[\overline{\fX_i} = \Ann{S/J}{x_i} \quad \text{ and } \quad \overline{\fY_i} = \Ann{S/J}{y_i}.\] Since $\overline{\fX_i}$ is a square-free monomial ideal, it is equal to its radical and
    \[\overline{\fX_i} = \bigcap_{\overline{\fq} \in \Supp[S/J]{\overline{x_i}}} \overline{\fq} = \bigcap_{J + \fX_i \subseteq \fq \subset S} \overline{\fq}  \] 
     where $\Supp[S/J]{M}$ is the set of all primes $\fq$ in $S/J$ such that the localization of $M$ at $\fq$ is non-zero, i.e., $M_{\fq} \neq 0$.

    Since we are intersecting over all primes containing $J + \fX_i$, we may assume without loss of generality that $\fq$ is minimal over $J + \fX_i$. Since $J + \fX_i$ is monomial, $\fq$ is also monomial, meaning $x_i - y_i \in \fq$ if and only if $x_i, y_i \in \fq$. Combining the fact that $\fq$ is minimal and $x_i$ does not divide any generators of $J + \fX_i$, we find $x_i \notin \fq$ and thus $x_i - y_i \notin \fq$. Now consider $s \in S$ such that $s(x_i - y_i) \in J \subset \fq$. Since $q$ is prime, we find $s \in \fq$, allowing us to conclude that for all $\overline{s} \in S/J$ satisfying $\overline{s}(\overline{x_i - y_i}) = \overline{0}$ in $S/J$, we have $\overline{s} \in \overline{\fq}$. This tells us that $\overline{x_i}$ and $\overline{y_i}$ correspond to distinct equivalence classes of $(S/J)_{\overline{\fq}}$ and thus $(\overline{x_i - y_i})_{\overline{\fq}} \neq 0$. This demonstrates that
    \[\Ann{S/J}{\overline{x_i - y_i}} \subseteq \bigcap_{\overline{\fq} \in \Supp[S/J]{\overline{x_i}}} \overline{\fq} = \overline{\fX_i}.\]
    An analogous argument shows $\Ann{S/J}{\overline{x_i - y_i}} \subseteq \overline{\fY_i}.$
    Combining these two results gives
    \[\Ann{S/J}{\overline{x_i - y_i}} \subseteq \overline{\fX_i} \cap \overline{\fY_i} = \overline{\fX_i \cap \fY_i} \subseteq \overline{J} = 0,\] that is, $\overline{x_i - y_i}$ is a non-zero divisor in $S/J$. Moreover, $\overline{x_i - y_i}$ is not a unit in $S/J$ since the only unit is 1. Thus $\overline{x_i-y_i}$ is a regular element in $S/J$, and $x_i - y_i$ is $(S/J)$-regular for all $1 \leq i \leq n$.

    Specifically, $x_1 - y_1$ is $(S/J)$-regular. For $i > 1$, set
    \[ M_i := \frac{\frac{S}{J}}{(x_1 - y_1, \ldots, x_{i-1} - y_{i-1}) \frac{S}{J}} \cong \frac{\F_2[x_1,\ldots,x_n,y_i,\ldots,y_n]}{J_{(i)}} \]
    where $J_{(i)}$ is the image of $J$ with each $y_1,\ldots,y_{i-1}$ replaced with $x_1, \ldots, x_{i-1}$, respectively. The proof that $\overline{x_i - y_i}$ is a regular element in $M_i$, viewed as a ring, is identical to the proof that $\overline{x_i - y_i}$ is a regular element in $S/J$. This shows that $x_i - y_i$ is $M_i$-regular for each $i > 1$, implying that the sequence $x_1 - y_1, \ldots, x_n - y_n$ is $(S/J)$-regular.
\end{proof}

We now prove Theorem \ref{thm:reg2indpierced}.

\begin{proof}[Proof of Theorem \ref{thm:reg2indpierced}]
    By our hypotheses, every generator of the canonical form of $J$ is of the form $x_ix_j$ or $x_iy_j$ for some $i,j$. If for some fixed $i$ and $j$, we have both $x_ix_j$ and $x_iy_j$, then the canonical form has a generator dividing $x_i$ by \cite{GellerRG}*{Algorithm 5.9}, contradicting our hypothesis that the canonical form consists of quadratics. If for some fixed $i$ and $j$, we have both $x_iy_j$ and $x_jy_i$, then $U_i=U_j$, contradicting our other hypothesis. Hence for any pair $i,j$, we have at most one generator chosen from the set $\{x_ix_j,x_iy_j,x_jy_i\}$.
    
    Consequently, we can draw the general relationship graph $\mathcal{G}(\CC)$ of $\CC$ as in \cite{curry2022recognizing}, with a vertex for each neuron and an edge connecting vertices $i$ and $j$ if $J$ has no generator from the set $\{x_ix_j,x_iy_j,x_jy_i\}$. The neural code $\CC$ is inductively pierced if and only if this graph is chordal by Theorem \ref{thm:IpChar}.

    If every generator of the canonical form is of the form $x_ix_j$, then $J$ is the edge ideal of the complement of $\mathcal{G}(\CC)$. By \cite{froberg}*{Theorem 1} as interpreted in \cite{banerjee}*{Theorem 6}, $\reg{J}=2$ if and only if $\mathcal{G}(\CC)$ is chordal, i.e. if and only if $\CC$ is inductively pierced.
    
    Otherwise, there exists a $y_i$ that divides one of the generators of $J$. In this case, set $S'=\F_2[x_1,\ldots,x_n]$ viewed as an $S$-module via the isomorphism with $S/(x_1-y_1,x_2-y_2,\ldots,x_n-y_n)$. By Lemma \ref{lem:indRegSeq}, $x_1-y_1,\ldots,x_n-y_n$ forms a $S/J$-regular sequence on $S/J$. By Corollary \ref{cor:freeresregseq}, tensoring a minimal free resolution of $J$ with this quotient gives a minimal free resolution of the ideal $J \otimes_S S'$ (over $S'$), where the degrees of the maps don't change. Hence $J$ and $J \otimes_S S'$ have the same regularity. This ideal will have the same generators as $J$, but with every $y_i$ replaced by the corresponding $x_i$. As a result, the new list of generators is in canonical form by \cite{GellerRG}*{Theorem 5.5}. The corresponding neural code $\CC'$ will have the same general relationship graph as $\mathcal{G}(\CC)$.
    
    By the reasoning in the all $x$'s case, $J \otimes_S S'$ has regularity 2 over $S'$ if and only if $\mathcal{G}(\CC')=\mathcal{G}(\CC)$ is chordal, and hence if and only if $\CC$ is inductively pierced. Thus $J$ has regularity 2 if and only if $\CC$ is inductively pierced.
\end{proof}

\begin{rem}\label{rem:linear}
    Note that if $J$ is the \pci for some inductively pierced code $\CC$, then it is generated by quadratics and, by Theorem \ref{thm:reg2indpierced}, has regularity 2. This tells us that the minimal free resolution of $J$ over $S$ is linear; that is, if $F$ is the minimal free resolution $J$ over $S$ then the entries of any matrix representation of the differentials $\partial^{F}$ are either 0 or linear terms.
\end{rem}

 This gives detailed information about the Betti and graded Betti numbers of polarized neural ideals.

\begin{rem}
    For our purposes, we are interested in the Betti numbers of $S/J$ instead of the Betti numbers of $J$. Fortunately, one can easily translate between the graded Betti numbers of $S/J$ and $J$ by observing that
\[\betti[S]{i,j}{S/J} = \begin{cases} 1 & i = 0 = j \\ \betti[S]{i-1,j}{J} & i > 0 \\ 0 & \text{otherise} \end{cases}.\] In the case where $J$ is the \pci of an inductively pierced code and $i > 0$, we find \[\betti[S]{i,j}{S/J} = \begin{cases} \betti[S]{i-1,i+1}{J} & j = i + 1 \\ 0 & j \neq i + 1 \end{cases}. \] As a consequence, if $F$ is the minimal free resolution of $S/J$ over $S$, then \[F_i = \begin{cases} S[0] & i = 0 \\ S[-(i+1)]^{\betti[S]{i,i+1}{S/J}} & i > 0 \\ 0 & i < 0 \end{cases}.\]
\end{rem} 

In order to distinguish more properties of \pci s, we extend these observations to a finer grading on the ring $S$, namely $\deg x=(1,0)$ and $\deg y=(0,1)$, and compute the \textit{multigraded Betti numbers}, defined below. See \cite{millersturmfels}*{Chapter 8} for general versions of the material below and more discussion of multigraded polynomial rings and Betti numbers.

\begin{defn} \label{def:multigradedbetti}
    Let $\F$ be any field. Assign $S=\F[x_1,\ldots,x_n,y_1,\ldots,y_n]$ a multigrading in which $\deg(x_i)=(1,0)$ and $\deg(y_i)=(0,1)$. Let $J \subseteq S$ be a monomial ideal (which is inherently homogeneous under the multigrading of $S$). Let $F_\bullet$ be a minimal multigraded free resolution of $S/J$.  The \textit{multigraded Betti numbers} of $S/J$ record the multigraded pieces of the $F_w$, i.e. if $F_w=\oplus_{u,v} S[-(u,v)]^{t_{w,u,v}}$, then
    \[\beta_{w,u,v}^S(S/J):=t_{w,u,v}.\]

    If $J$ is the \pci of a neural code $\CC$, we may refer to the multigraded Betti numbers of $\CC$.
\end{defn}

\begin{rem}\label{rem:wuvTwist}
    Suppose $J$ is the \pci of some inductively pierced neural code $\CC$. If $F$ is the minimal free resolution of $S/J$ over $S$ then $F_0 = S[(0,0)]$ and, for $w > 0$, \[F_w = \bigoplus_{u + v = w + 1} S[-(u,v)]^{\betti[S]{w,u,v}{S/J}}.\] Moreover, for $w > 0$ the Betti numbers satisfy \[\betti[S]{w,w+1}{S/J} = \sum_{u + v = w + 1} \betti[S]{w,u,v}{S/J} \] and $\betti[S]{w,j}{S/J} = 0$ when $j \neq w + 1$.
\end{rem}

In the next remark, we demonstrate how the \pci of an inductively pierced code is constructed recursively. 

\begin{rem}\label{rem:pierceIdeal}
Combining Definition~\ref{def:kpiercing}, Theorem~\ref{thm:canonpolarizes}, and Lemma~\ref{lem:polarizedmingens}, we get the following polarized version of \cite{curry2022recognizing}*{Lemma 16}: if $J_n$ is the \pci of an inductively pierced code $\CC$ as in Notation \ref{not:piercingideal},
\begin{align}
  J_{n} = J_{n-1} + (x_ix_{n} : i \in [n-1] \backslash \tau) + (y_j x_{n} : j \in \sigma).  \label{eqn:canonRelate}
\end{align}
In other words, when $n$ is a $k$-piercing of $\CC \backslash (n)$, $J_n$ is constructed from $J_{n-1}$ by adding $x_ix_n$ for any $i$ such that $U_n \cap U_i=\emptyset$, and $y_jx_n$ for any $j$ such that $U_n \subseteq U_j$.

\end{rem}

\begin{defn}
\label{def:piercingideal} 
    We define the \textit{piercing ideal} for neuron $n$ to be \[\fp_{n} := (x_i : i \in [n-1] \backslash \tau) + (y_j : j \in \sigma).\]
\end{defn}

\begin{rem}\label{rem:FiberIdeal}
We note that $\fp_{n}$ is prime, is resolved by the Koszul complex $K^S(\fp_{n})$, and allows us to rewrite Equation \eqref{eqn:canonRelate} as 
\begin{align}
    J_{n} &= J_{n-1} + (x_{n})\fp_{n}. \label{eqn:canonIdeal}
\end{align}
\end{rem}

\begin{prop}\label{prop:inP}
Consider a neural code $\CC \subseteq 2^{[n]}$, and assume that $n$ is a $k$-piercing of $\CC$ and $J_{n-1}$ is as in Notation~\ref{not:piercingideal}. Defining $\fp_{n}$ as in Definition~\ref{def:piercingideal}, one has $J_{n-1} \subset \fp_{n}$.
\end{prop}

\begin{proof}
Suppose $m=\prod_{i \in \alpha} x_i \prod_{j \in \beta} y_j$ is a generator of $J_{n-1}$ with no factor in $\fp_{n}$, i.e. for all $\ell \in \alpha \cup \beta$, $U_{\ell} \cap U_{n} \ne \emptyset$ and $U_{n} \subsetneq U_{\ell}$. We will show that this is impossible.
Note that $\alpha \cap \beta=\emptyset$.

For all $i \in \alpha$, we have $U_i \cap U_{n} \ne \emptyset$ and thus $\alpha \subseteq \tau$ in the sense of Definition~\ref{def:kpiercing}. Hence $\left(\bigcap_{i \in \alpha} U_i\right) \cap U_{n} \ne \emptyset.$ Since $U_{n} \subsetneq U_j$ for any $j \in \beta$, $j \not\in \sigma$ in the sense of Definition~\ref{def:kpiercing}. 
Since $\bigcap_{i \in \alpha} U_i \subseteq \bigcup_{j \in \beta} U_j$, we must have
\begin{equation}
\label{eq:containment}
    \left(\bigcap_{i \in \alpha} U_i\right) \cap U_{n} \subseteq \bigcup_{j \in \beta} U_j.
\end{equation}
Choose $\beta' \subseteq \beta$ minimal such that the above relation still holds with $\beta$ replaced by $\beta'$. Then $U_{n} \cap U_j \ne \emptyset$ for every $j \in \beta'$. Hence $\beta' \subseteq \tau \backslash \sigma$. Then we must have a point contained in $\bigcap_{i \in \alpha} U_i \cap U_{n}$ but not contained in $U_j$ for any $j \in \beta'$, which is a contradiction of Equation \ref{eq:containment}.
\end{proof}

\begin{lemma}\label{lem:SES}
    Under the conditions of Proposition \ref{prop:inP}, the sequence
    \[
    \xymatrix{
    0 \ar[r] & \frac{S}{\fp_n} \ar[r]^-{\cdot x_n} & \frac{S}{J_n} \ar[r] & \frac{S}{J_{n - 1} + (x_n)} \ar[r] & 0
    }
    \]
    is exact.
\end{lemma}

\begin{proof}
    To start, we claim that $\fp_n = (J_n :_S x_n)$. Since $x_n$ is regular on $S/J_{n-1}$, we have $(J_{n-1}:_S x_n) = J_{n-1}$. Combined with Remark \ref{rem:FiberIdeal}, we have \[(J_n :_S x_n) = (J_{n-1} :_S x_n) + (x_n \fp_n :_S x_n) = J_{n-1} + \fp_n = \fp_n\] where the last equality follows from $J_{n-1} \subset \fp_n$, i.e., Proposition \ref{prop:inP}.

    Next, by Proposition \ref{prop:inP}, we have \[J_n + (x_n) = J_{n-1} + x_n \fp_n + (x_n) = J_{n-1} + (x_n).\] Substituting both $\fp_n = (J_n :_S x_n)$ and $J_n + (x_n) = J_{n-1} + (x_n)$ into the short exact sequence \[
    \xymatrix{
    0 \ar[r] & \frac{S}{(J_n :_S x_n)} \ar[r]^-{\cdot x_n} & \frac{S}{J_n} \ar[r] & \frac{S}{J_n + (x_n)} \ar[r] & 0
    }
    \] provides the desired result.
\end{proof}

\begin{notation}
    If $n$ is a $k$-piercing of $C$, then we denote the minimal free resolutions of $S/J_{n-1}$ and $S/J_{n}$ by $F^{(n-1)}$ and $F^{(n)}$, respectively.
\end{notation}

\begin{lemma}\label{lem:tensorXn}
    Let $K^S(x_n)$ denote the Koszul complex on $x_n$. If $G^{(n)}$ is the minimal free resolution of $S/(J_{n-1} + (x_n))$ over $S$, then $G^{(n)} \cong F^{(n-1)} \otimes_S K^S(x_n)$. Consequently, 
    \begin{align*}
        G_0^{(n)} &= S[(0,0)], \\
        G_1^{(n)} &= S[-(1,0)] \oplus \left( \bigoplus_{u + v = 2} S[-(u,v)]^{\betti[S]{1,u,v}{S/J_{n-1}}} \right), \text{ and, for } w > 1, \\
        G_w^{(n)} &= \bigoplus_{u + v = w + 1} S[-(u,v)]^{\betti[S]{w,u,v}{S/J_{n-1}} + \betti[S]{w-1,u-1,v}{S/J_{n-1}}}.
    \end{align*}
\end{lemma}

\begin{proof}
    Since $x_n$ is regular over $S/J_{n-1}$, by Theorem \ref{thm:CommRing} and Corollary \ref{cor:tensorRes}
    \[\frac{S}{J_{n-1} + (x_n)} \cong \frac{S}{J_{n-1}} \otimes_S \frac{S}{(x_n)}\] is minimally resolved by $F^{(n-1)} \otimes_S K^S(x_n)$.
    That is, \[G^{(n)} \cong F^{(n-1)} \otimes_S K^S(x_n).\] 
    The claim now follows by analyzing the graded components of the above complex.
\end{proof}

\begin{lemma}\label{lem:fpnMultiGrade}
    Under the conditions of Proposition \ref{prop:inP}, if $n$ contained in $\ell$ of the other place fields, then \[\betti[S]{w,u,v}{S/\fp_n} = \begin{cases} \tbinom{n-1-k-\ell}{u} \tbinom{\ell}{v} & w = u + v \\ 0 & w \neq u + v. \end{cases}\]
\end{lemma}

\begin{proof}
    Since $n$ is a $k$-piercing contained in $\ell$ other place fields, it must be disjoint from $n - 1 - k - \ell$ place fields. Thus, $\fp_n = (\underline{x}, \underline{y})$ where $\underline{x}$ is a list of $n - 1 - k - \ell$ distinct $x_i$'s and $\underline{y}$ is a list of $\ell$ distinct $y_j$'s. It follows that $S / \fp_n$ is resolved by the Koszul complex, $K^S(\fp_n)$. As such, any basis element of the resolution with degree $-(u,v)$ must correspond to $u$ of $n - 1 - k - \ell$ of the $x$'s and $v$ of $\ell$ of the $y$'s. This gives our result.
\end{proof}

\begin{rem}\label{rem:adjustedSES}
    In the short exact sequence given by Lemma \ref{lem:SES}, the map $S/\fp_n \to S / J_n$ is given by multiplication by $x_n$. This map has degree $(1,0)$ in our multigrading, so to make the map homogeneous, we replace $S / \fp_n$ with $(S / \fp_n)[-(1,0)]$ to get the short exact sequence
    \[
    \xymatrix{
    0 \ar[r] & \frac{S}{\fp_n}[-(1,0)] \ar[r]^-{\cdot x_n} & \frac{S}{J_n} \ar[r] & \frac{S}{J_{n - 1} + (x_n)} \ar[r] & 0.
    }
    \]
    Since $S / \fp_n$ is minimally resolved by $K^S(\fp_n)$, we have $(S / \fp_n)[-(1,0)]$ is minimally resolved by $K^S(\fp_n) \otimes_S S[-(1,0)]$. 
    Using Lemma \ref{lem:fpnMultiGrade}, the minimal free resolution in homological degree $w$ is given by
    \[\bigoplus_{u + v = w} S[-(u,v)]^{\tbinom{n - 1 - k - \ell}{u} \tbinom{\ell}{v}} \otimes_S S[-(1,0)] \cong \bigoplus_{u + v = w} S[-(u+1,v)]^{\tbinom{n - 1 - k - \ell}{u} \tbinom{\ell}{v}}.\] By reindexing the $u$-component, we get the free $S$-module \[\bigoplus_{u + v = w + 1} S[-(u,v)]^{\tbinom{n - 1 - k - \ell}{u - 1} \tbinom{\ell}{v}}\] with multigraded Betti numbers given by 
    \[\betti[S]{w,u,v}{\frac{S}{\fp_n}[-(1,0)]} = \begin{cases} \tbinom{n-1-k-\ell}{u - 1} \tbinom{\ell}{v} & w + 1 = u + v \\ 0 & w + 1 \neq u + v. \end{cases}\]
\end{rem}

In the proof of the following theorem, we consider a long exact sequence in Tor associated to the short exact sequence from Remark \ref{rem:adjustedSES}. We note that since all the modules in the sequence are graded, the long exact sequence in Tor respects the same grading.

\begin{thm}\label{thm:inductMultiBetti}
    Consider the setting of Notation \ref{not:piercingideal}. If $\CC$ is an inductively pierced code such that $n$ is a $k$-piercing of $\CC \backslash \{n\}$ contained in $\ell$ other place fields, then the multigraded Betti numbers of $S/ J_n$ and $S / J_{n-1}$ over $S$ satisfy the recursive formulas
    \begin{align*}
    \betti{0,u,v}{\frac{S}{J_{n}}} &= \begin{cases} 1 & u = v = 0 \\ 0 & \text{otherwise} \end{cases} \\
    \betti{1,u,v}{\frac{S}{J_{n}}} &= \begin{cases} \tbinom{n - 1 - k - \ell}{u - 1} \tbinom{\ell}{v} + \betti{1,u,v}{\frac{S}{J_{n-1}}} & u+v = 2 \text{ and } u > 0 \\
     0 & \text{otherwise} \end{cases} \\
    \betti{w,u,v}{\frac{S}{J_{n}}} &= \begin{cases} \tbinom{n - 1 - k - \ell}{u - 1} \tbinom{\ell}{v} + \betti{w-1,u-1,v}{\frac{S}{J_{n-1}}} + \betti{w,u,v}{\frac{S}{J_{n-1}}} & \begin{matrix} u + v = w + 1 \\ \text{ and } u > 0 \end{matrix} \\ 0 & \text{otherwise} \end{cases}
    \end{align*}
\end{thm}

\begin{proof}
    We start by noting that the case of $w = 0$ is proven in Remark \ref{rem:wuvTwist}. To get the case of $w > 0$, we observe that if \[
    \xymatrix{
    0 \ar[r] & M \ar[r] & M' \ar[r] & M'' \ar[r] & 0 
    }
    \] is a short exact sequence of $S$-modules, then there is an associated long exact sequence in Tor given by \[
    \xymatrix{
     \cdots\!\ar[r] & \Tor[w]{S}{M}{N} \ar[r] & \Tor[w]{S}{M'}{N} \ar[r] & \Tor[w]{S}{M''}{N} \ar`r[rd]`[l]`^d[lll]`^r[dll][dll] & \\
    & \Tor[w-1]{S}{M}{N} \ar[r] & \Tor[w-1]{S}{M'}{N} \ar[r] & \Tor[w-1]{S}{M''}{N} \ar[r] & \cdots
    }
    \] for any $S$-module $N$.

    For our purposes, we set $M = \frac{S}{\fp_n}[-(1,0)]$, $M' = \frac{S}{J_n}$, $M'' = \frac{S}{J_{n-1} + (x_n)}$, and $N = \frac{S}{(x_1,\ldots,x_n,y_1,\ldots,y_n)} \cong \F_2$. In this case, we obtain
    \begin{align*}
        \Tor[w]{S}{\frac{S}{\fp_n}}{\F_2} &\cong \Tor[w]{S}{K^S(\underline{x}) \otimes_S K^S(\underline{y}) \otimes_S S[-(1,0)]}{\F_2} \\
         &= H_w\left( \left(K^S(\underline{x}) \otimes_S K^S(\underline{y}) \otimes_S S[-(1,0)]\right) \otimes_S \F_2 \right) \\
         &= \left(K^S(\underline{x}) \otimes_S K^S(\underline{y})_w \otimes_S S[-(1,0)]\right) \otimes_S \F_2 \\
         &= \bigoplus_{u + v = w + 1} S[-(u,v)]^{\tbinom{n - 1 - k - \ell}{u - 1} \tbinom{\ell}{v}} \otimes_S \F_2 \\
         &= \bigoplus_{u + v = w + 1} \F_2[-(u,v)]^{\tbinom{n - 1 - k - \ell}{u - 1} \tbinom{\ell}{v}}.
    \end{align*}
    Moreover, since Tor respects gradings, i.e., it is also graded, we can decompose the Tor module into its $(u,v)$-graded components. For $S / \fp_n$, we find \[\Tor[w]{S}{\frac{S}{\fp_n}}{\F_2}_{(u,v)} = \begin{cases} \F_2[-(u,v)]^{\tbinom{n - 1 - k - \ell}{u - 1} \tbinom{\ell}{v}} & u + v = w + 1 \\ 0 & u + v \neq w + 1 \end{cases}.\]

    Using the same argument in conjunction with Remark \ref{rem:wuvTwist} yields \[\Tor[0]{S}{\frac{S}{J_n}}{\F_2}_{(u,v)} = \begin{cases} \F_2[(0,0)] & u = v = 0 \\ 0 & \text{otherwise} \end{cases}\] and, for $w > 0$, \[\Tor[w]{S}{\frac{S}{J_n}}{\F_2}_{(u,v)} = \begin{cases} \F_2[-(u,v)]^{\betti[S]{w,u,v}{\frac{S}{J_n}}} & u + v = w + 1 \\ 0 & \text{otherwise} \end{cases}.\] Similarly argued, Lemma \ref{lem:tensorXn} lets us conclude that \[\Tor[0]{S}{\frac{S}{J_{n-1} + (x_n)}}{\F_2}_{(u,v)} = \begin{cases} S[(0,0)] & u = v = 0 \\ 0 & \text{otherwise} \end{cases},\] \[\Tor[1]{S}{\frac{S}{J_{n-1} + (x_n)}}{\F_2}_{(u,v)} = \begin{cases} \F_2[(-1,0)] & u = 1, v = 0 \\ \F_2[-(u,v)]^{\betti[S]{1,u,v}{S/J_{n-1}}} & u + v = 2 \\ 0 & \text{otherwise} \end{cases},\] and, for $w > 1$, \[\Tor[w]{S}{\frac{S}{J_{n-1} + (x_n)}}{\F_2}_{(u,v)} = \begin{cases} \F_2[-(u,v)]^{\betti[S]{w-1,u-1,v}{\frac{S}{J_{n-1}}} + \betti[S]{w,u,v}{\frac{S}{J_{n-1}}}} & u + v = w + 1 \\ 0 & \text{otherwise} \end{cases}.\]

    We compute these decompositions since any $S$-module homomorphism $\F_2[-(u,v)] \to \F_2[-(u',v')]$ is zero if either $u \neq u'$ or $v \neq v'$. This means the long exact sequence in Tor also decomposes into $(u,v)$ graded pieces. We proceed by considering various possibilities for $(u,v)$.

    Aside from the cases $(u,v)=(0,0)$ and $(u,v)=(1,0)$, we have all Betti numbers $\betti[S]{w,u,v}{\cdot} = 0$ for all three quotient rings, except possibly when $u + v = w + 1$. Thus, if $(u,v) \notin \{(0,0), (1,0)\}$, then the long exact sequence in Tor reduces to
    \[
    0 \to \Tor[w]{S}{\frac{S}{\fp_n}}{\F_2}_{(u,v)}\!\to\!\Tor[w]{S}{\frac{S}{J_n}}{\F_2}_{(u,v)}\!\to\!\Tor[w]{S}{\frac{S}{J_{n-1} + (x_n)}}{\F_2}_{(u,v)} \to 0
    \] which is short exact. Since the sequence is exact, the alternating sum of the ranks (as $\F_2$ vector spaces) must sum to 0. That is, the alternating sum of the corresponding multigraded Betti numbers must equal 0. In the case of $w = 1$, so when $u + v = 2$, we obtain
    \[0 = \tbinom{n - 1 - k - \ell}{u - 1} \tbinom{\ell}{v} - \betti[S]{1,u,v}{\frac{S}{J_n}} + \betti[S]{1,u,v}{\frac{S}{J_{n-1}}}\] which rearranges to the desired result with the observation that if $u \leq 0$, then \[\tbinom{n - 1 - k - \ell}{u - 1} = 0 = \betti[S]{1,u,v}{\frac{S}{J_{n-1}}}.\]

    The case of $w > 1$ yields the expression
    \[0 = \tbinom{n - 1 - k - \ell}{u - 1} \tbinom{\ell}{v} - \betti[S]{w,u,v}{\frac{S}{J_n}} + \betti[S]{w,u,v}{\frac{S}{J_{n-1}}} + \betti[S]{w - 1,u - 1,v}{\frac{S}{J_{n-1}}}.\] As in the previous case, if $u \leq 0$, then all terms vanish.

    To complete the proof, we note $(u,v) = (0,0)$ vanishes everywhere except for homological degree $w=0$, which was addressed at the beginning of the proof. The case of $(u,v) = (1,0)$ follows from Remark \ref{rem:wuvTwist} since a consequence of that lemma is that $\betti[S]{w,u,v}{S/J_n} = 0$ if $u + v = 1$.
\end{proof}

\begin{cor}\label{cor:gradedBetti}
    Under the same condition as Theorem \ref{thm:inductMultiBetti}, the graded Betti numbers of $S / J_n$ and $S / J_{n-1}$ over $S$ satisfy
    \begin{align*}
    \betti{0,i}{\frac{S}{J_{n}}} &= \begin{cases} 1 & i = 0 \\ 0 & \text{otherwise} \end{cases} \\
    \betti{1,i}{\frac{S}{J_{n}}} &= \begin{cases} \tbinom{n - 1 - k}{1} + \betti{1,2}{\frac{S}{J_{n-1}}} & i = 2 \\
     0 & \text{otherwise} \end{cases} \\
    \betti{w,i}{\frac{S}{J_{n}}} &= \begin{cases} \tbinom{n - 1 - k}{w} + \betti{w-1,i-1}{\frac{S}{J_{n-1}}} + \betti{w,i}{\frac{S}{J_{n-1}}} & j = w + 1 \\ 0 & \text{otherwise} \end{cases}
    \end{align*}
\end{cor}

\begin{proof}
    This follows from the equation \[\betti[S]{w,i}{S/J_n} = \sum_{u + v = i} \betti[S]{w,u,v}{S/J_n}\] and the Chu-Vandermonde Identity for binomial coefficients.
\end{proof}

\begin{cor}\label{cor:regularBetti}
    Under the same condition as Theorem \ref{thm:inductMultiBetti}, the Betti numbers of $S / J_n$ and $S / J_{n-1}$ over $S$ satisfy
    \[\betti{w}{S/J_{n}} = \begin{cases}
        1 & w = 0 \\
        \betti{1}{S/J_{n-1}} + \tbinom{n - 1 - k}{1} & w = 1 \\
        \betti{w}{S/J_{n-1}} + \betti{w - 1}{S/J_{n-1}} + \tbinom{n - 1 - k}{w} & w>1.
    \end{cases}\]
\end{cor}

\begin{proof}
    This is an immediate consequence of Corollary \ref{cor:gradedBetti} combined with the fact $J_n$ is generated by quadratics and $\reg{J_n} = 2$.
\end{proof}

\section{Main Betti numbers results}
\label{sec:mainresults}

In this section we leverage the results of Section~\ref{sec:fiberproduct} to give formulas for the Poincar\'{e} series, Betti numbers, projective dimensions, and regularity of polarized neural ideals of inductively pierced codes.

\begin{notation}
    From now on, we will compute Betti numbers over the ring $S=\F_2[x_1,\ldots,x_n,y_1,\ldots,y_n]$ unless otherwise indicated.
\end{notation}

We provide an example motivating our study of graded and multi-graded Betti numbers.

\begin{example}\label{ex:multigraded}
    While non-graded Betti numbers can often distinguish ideals from each other, they are insufficient to even detect whether a code is inductively pierced. Consider the \pci $J_1=(x_1x_3,x_2x_4)$. This has the same Betti numbers as $J_2=(x_1x_4,x_3x_4)$, but only the latter is inductively pierced (see Figure \ref{fig:rings}). Here the two neural codes can be distinguished by their graded Betti numbers, or by their regularity (3 for the first ideal, 2 for the second). This incentivizes us to consider graded Betti numbers rather than only non-graded Betti numbers in order to detect inductively pierced codes.

    In this paper we go further and consider multi-graded Betti numbers under the multi-grading where $\deg(x_i)=(1,0)$ and $\deg(y_i)=(0,1)$. This allows us to differentiate between firing fields that do not intersect and firing fields where one is contained in the other. For example, $J_3=(x_1x_4,x_4y_3)$ has the same graded Betti numbers as $J_2$ above. However, for $J_2$, $U_3 \cap U_4=\emptyset$, whereas for $J_3$, $U_4 \subseteq U_3$ (see Figure \ref{fig:rings}). Indeed, their multigraded Betti numbers are distinct.
\end{example}

\begin{figure}
\includegraphics[width=2in]{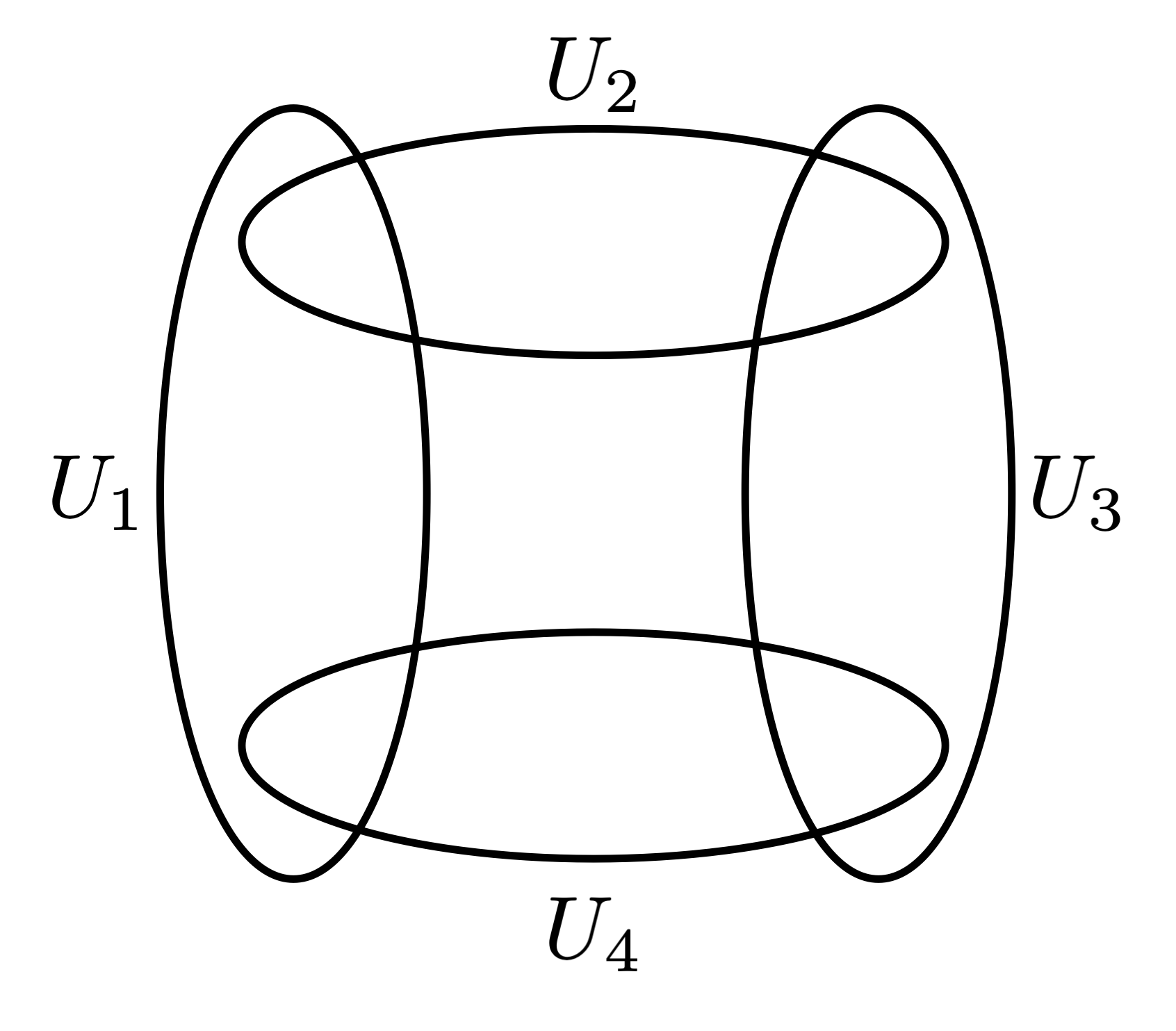} \\
\includegraphics[width=2in]{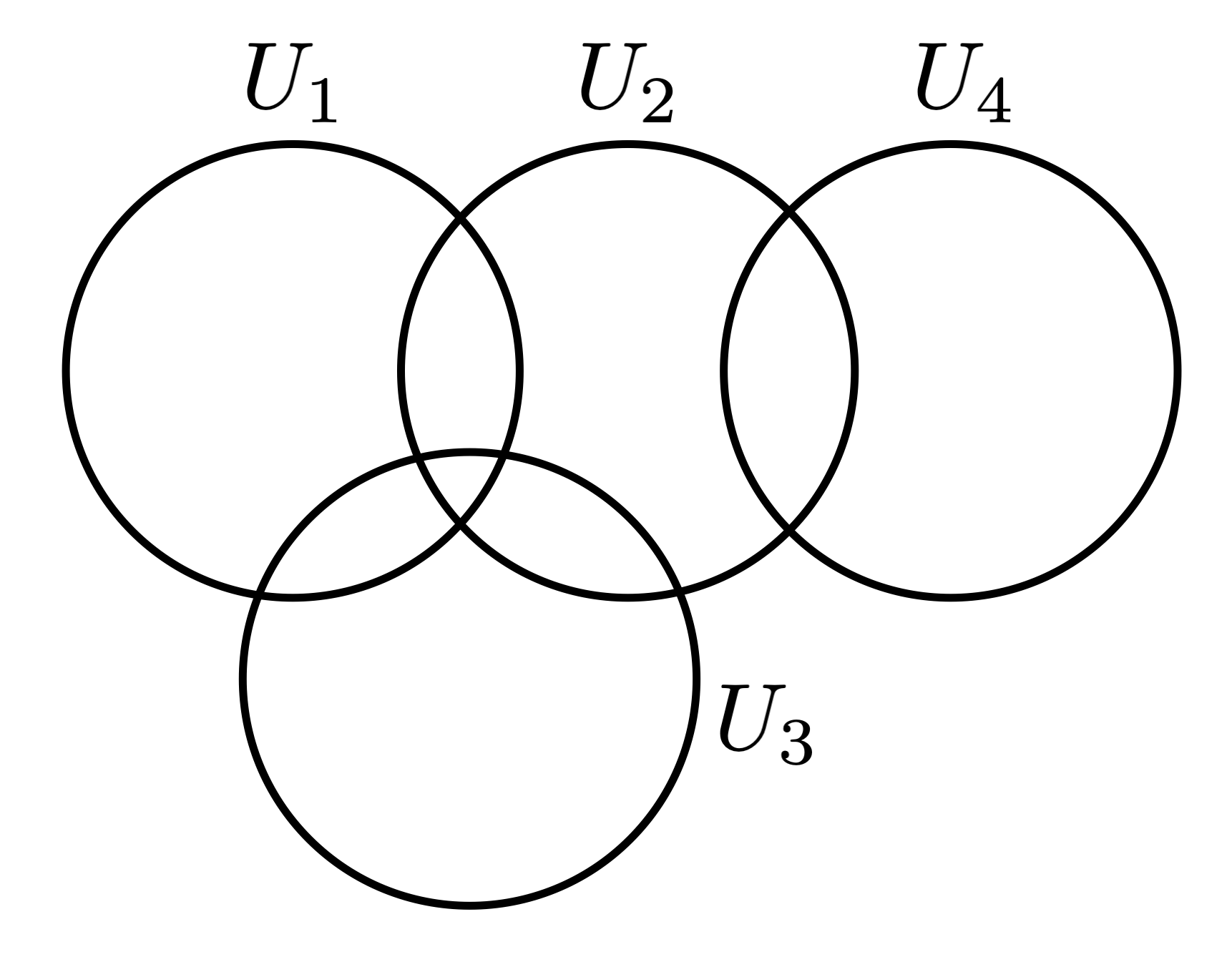}\hspace{0.5in} \includegraphics[width=1.6in]{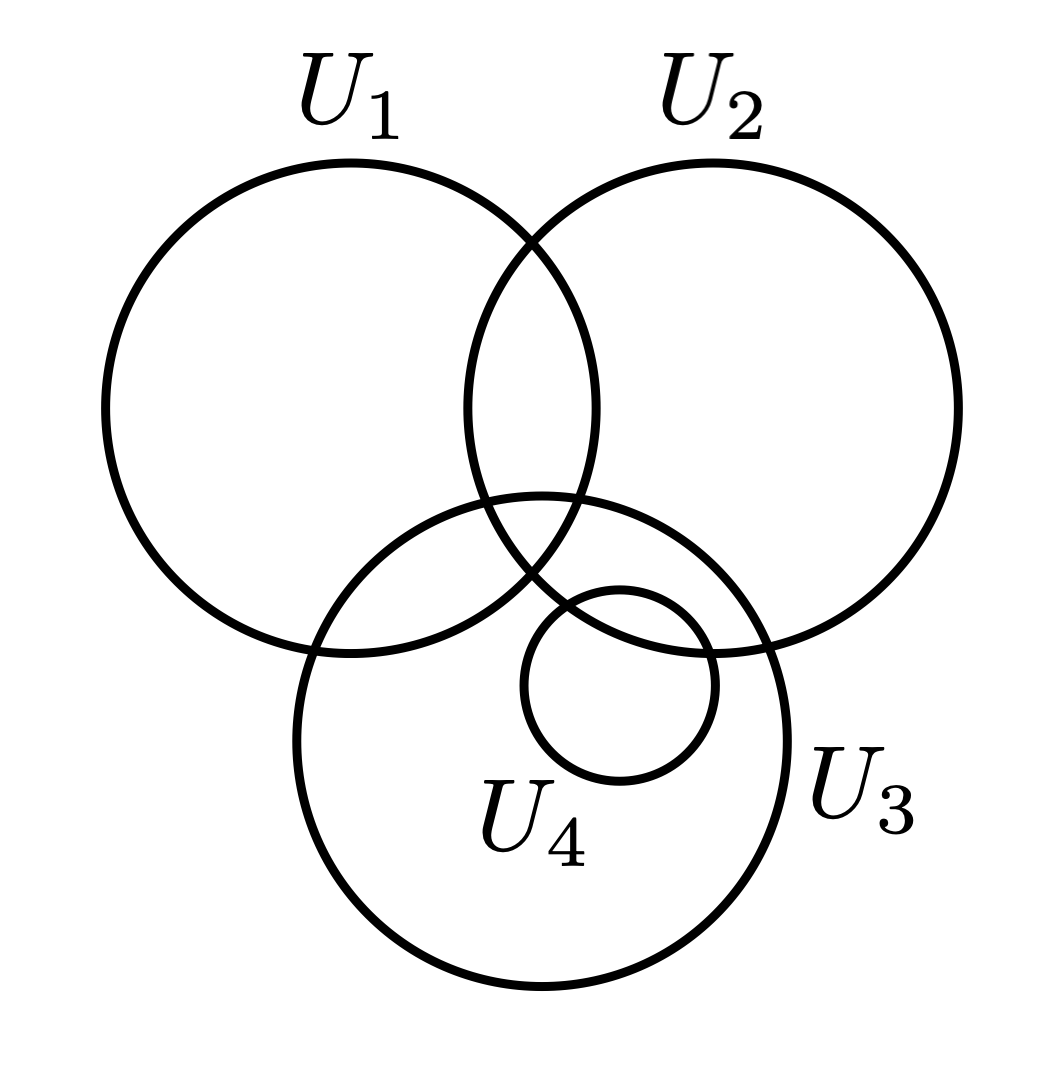}
\caption{Realizations of the codes associated to $J_1$ (top), $J_2$ (lower left), $J_3$ (lower right) from Example \ref{ex:multigraded}.}
\label{fig:rings}
\end{figure}

\begin{defn}\label{def:j_k}
    Let $C$ be an inductively pierced neural code on $n$ neurons. Given a piercing order of $C$, we define $j_k$ to be the number of $k$-piercings in the piercing order, i.e. with $\text{rank}[\sigma,\tau]=k$.
    
    We further define $j_{k,v}$ to be the number of $k$-piercings contained in $v$ other place fields in the piercing order, i.e. with $|\sigma|=v$ and $\text{rank}[\sigma,\tau]=k$ in the notation of Definition~\ref{def:kpiercing}.
\end{defn}

\begin{lemma}
\label{lem:piercingsum}
    Let $C$ be an inductively pierced neural code on $n$ neurons. For $0 \leq k < n$, we have \[j_k = \sum_{v=0}^{n - 1} j_{k,v} = \sum_{v=0}^{n - k - 1} j_{k,v}.\]
    In particular, a $k$-piercing can be contained in at most $n-k-1$ place fields.
    Moreover, we have \[n = \sum_{k = 0}^{n - 1} j_k.\]
\end{lemma}

\begin{proof}
    Since we have $n$ neurons, the first equality follows from that fact that any given place field is contained in no more than $n - 1$ other place fields.

    To get the second equality, consider $k$ and $v$ such that $j_{k,v} \neq 0$. This means there exists at least one $U_i$ creating a $k$-piercing contained in $v$ other placefields. For this to be possible, we must have \[1 + k + v \leq n.\] As a result, if $v > n - 1 - k$, then $j_{k,v} = 0$, which then gives the second equality.

    Lastly, the final summation follows from the fact that every neuron is a piercing and the largest possible piercing with $n$ neurons is an $n-1$-piercing.
\end{proof}

\begin{example}
    Continuing Example~\ref{ex:prcgorder}, we have:
    \[ j_{0,0} = 1 \quad j_{1,0} = 2 \quad j_{1,1} = 1 \quad j_{2,0} = 1.\]
    Or we could write
    \[ j_0 = 1 \quad j_1 = 3 \quad j_2 = 1.\]
\end{example}

\begin{rem} Although inductively pierced codes may have many possible piercing orders, we will show that the numbers $j_{k,v}$ are independent of the particular order chosen (see Theorem~\ref{thm:jkl}). This will also have interesting consequences for chordal graphs, see Section~\ref{sec:chordal}.
\end{rem}

\begin{rem}
    As noted earlier, in order to have a $k$-piercing, a neural code must have a $(k-1)$-piercing. As a result, given a piercing order, if $j_k=0$ for some $k$, then $j_u=0$ for all $u >k$ and $C$ is at most $(k-1)$-inductively pierced. Similarly, if $j_{k,v}=0$ for some $k$, then $j_{u,v}=0$ for all $u>k$.
\end{rem}

From Theorem \ref{thm:inductMultiBetti}, we can compute the Betti numbers of an inductively pierced code.

\begin{notation}
    In what follows, $\delta_{a,b}$ will denote the Kronecker delta function, i.e. 
    \[
    \delta_{a,b}=\begin{cases} 1 & \text{if } a=b, \\ 0 & \text{else.} \\ \end{cases}
    \]
\end{notation}

\begin{thm}\label{thm:xybettis}
    Suppose $\CC$ is an inductively pierced neural code on $n$ neurons, and $J_n \subset S$ is the \pci of $\CC$. Given a piercing order of $\CC$ and corresponding $j_{k,\ell}$, if $w > 0$ 
    then \[\betti{w,w+1-v,v}{S/J_n} = \sum_{k = 0}^{n - 1} \sum_{\ell = 0}^{n - 1} \left(j_{k,\ell} - \delta_{\ell,0}\right) \binom{n-1-k-\ell}{w - v} \binom{\ell}{v}.\] Moreover, $\betti{0,0,0}{S/J_n} = 1$, and all other Betti numbers are 0.
\end{thm}

\begin{proof}
    We proceed by induction on $n$. For the case of $n = 1$, we have $J = (0)$ and so $S / J = S$. Thus, \[\betti[S]{w,u,v}{S/J} = \begin{cases} 1 & w = u = v = 0 \\ 0 & \text{ otherwise} \end{cases}.\] For the case of $w > 0$, we note that $k = \ell = 0 = n - 1$, so $j_{k,\ell} - \delta_{\ell,0} = 0$, causing the theorized formula for $\betti[S]{w,w+1-v,v}{S/J}$ also to vanish.

    Now suppose the formula holds for some $n-1$ with $n > 0$ and consider $n$ such that $n$ is a $k$-piercing contained in $\ell$ other place fields. Let $J_n$ denote the \pci   for $\CC$ and $J_{n-1}$ the \pci for $\CC \backslash \{n\}$.  By Theorem \ref{thm:inductMultiBetti}, we have the desired vanishing and $\betti[S]{0,0,0}{S/J_n} = 1$. Thus, we turn our attention to the case of $w > 0$.

    If $w = 1$ and $u + v = w + 1$, then either $v = 0$ or $v = 1$. In the case of $v = 0$, Theorem \ref{thm:inductMultiBetti} and our induction hypothesis yield
    \begin{align*}
        \betti[S]{1,2,0}{\frac{S}{J_n}} &= \tbinom{n - 1 - k - \ell}{1}\tbinom{\ell}{0} + \betti[S]{1,2,0}{\frac{S}{J_{n-1}}} \\
         &= \tbinom{n - 1 - k - \ell}{1} + \sum_{a = 0}^{n - 2} \sum_{b = 0}^{n - 2} \left(j_{a,b}' - \delta_{b,0}\right) \tbinom{n-2-a-b}{1} \tbinom{b}{0} \\
         &= \tbinom{n - 1 - k - \ell}{1} + \sum_{a = 0}^{n - 2} \sum_{b = 0}^{n - 2} \left(j_{a,b}' - \delta_{b,0}\right) \tbinom{n-2-a-b}{1} \\
         &= \tbinom{n - 1 - k - \ell}{1} + \sum_{a = 0}^{n - 2} \sum_{b = 0}^{n - 2} \left(j_{a,b}' - \delta_{b,0}\right) \left[\tbinom{n - 1 - a - b}{1} - \tbinom{n - 2 - a - b}{0}\right] \\
         &= \tbinom{n - 1 - k - \ell}{1} + \sum_{a = 0}^{n - 2} \sum_{b = 0}^{n - 2} \left(j_{a,b}' - \delta_{b,0}\right) \tbinom{n - 1 - a - b}{1} - \sum_{a = 0}^{n - 2} \sum_{b = 0}^{n - 2} \left(j_{a,b}' - \delta_{b,0}\right) \\
         &= \tbinom{n - 1 - k - \ell}{1} + \sum_{a = 0}^{n - 2} \sum_{b = 0}^{n - 2} \left(j_{a,b}' - \delta_{b,0}\right) \tbinom{n - 1 - a - b}{1} - \underset{=0}{\underbrace{\sum_{a = 0}^{n - 2} \left(j_{a} - 1\right)}} \\
         &= \tbinom{n - 1 - k - \ell}{1} + \sum_{a = 0}^{n - 2} \sum_{b = 0}^{n - 2} \left(j_{a,b}' - \delta_{b,0}\right) \tbinom{n - 1 - a - b}{1}
    \end{align*}
    If $n$ is contained in all other place fields, then $\ell = n-1$ and $n$ must be a 0-piercing, so $k = 0$. This yields $\tbinom{n - 1 - k - \ell}{1} = \tbinom{0}{1} = 0$ and so \[\betti[S]{1,2,0}{\frac{S}{J_n}} = \sum_{a = 0}^{n - 2} \sum_{b = 0}^{n - 2} \left(j_{a,b}' - \delta_{b,0}\right) \tbinom{n - 1 - a - b}{1}.\] Now noting that if either $a = n - 1$ or $b = n - 1$, then $\tbinom{n - 1 - a - b}{1} = 0$ and thus it must be vacuously true that
    \begin{align*}
        \betti[S]{1,2,0}{\frac{S}{J_n}} &= \sum_{a = 0}^{n - 1} \sum_{b = 0}^{n - 1} \left(j_{a,b}' - \delta_{b,0}\right) \tbinom{n - 1 - a - b}{1} \\
         &= \sum_{a = 0}^{n - 1} \sum_{b = 0}^{n - 1} \left(j_{a,b}' - \delta_{b,0}\right) \tbinom{n - 1 - a - b}{1} \tbinom{b}{0} \\
         &= \sum_{a = 0}^{n - 1} \sum_{b = 0}^{n - 1} \left(j_{a,b} - \delta_{b,0}\right) \tbinom{n - 1 - a - b}{1} \tbinom{b}{0}.
    \end{align*}
    Using the same arguments, the case of $w = u = v = 1$ yields
    \begin{align*}
        \betti[S]{1,1,1}{\frac{S}{J_n}} &= \tbinom{\ell}{1} + \sum_{a = 0}^{n - 2} \sum_{b = 0}^{n - 2} \left(j_{a,b}' - \delta_{b,0}\right) \tbinom{b}{1} \\
         &= \tbinom{\ell}{1} + \sum_{a = 0}^{n - 2} \sum_{b = 0}^{n - 2} \left(j_{a,b}' - \delta_{b,0}\right) b \\
         &= \tbinom{\ell}{1} + \sum_{a = 0}^{n - 2} \sum_{b = 0}^{n - 2} b j_{a,b}'  - \sum_{a = 0}^{n - 2} \sum_{b = 0}^{n - 2} \underset{=0}{\underbrace{b \delta_{b,0}}} \\
         &= \tbinom{\ell}{1} + \sum_{a = 0}^{n - 2} \sum_{b = 0}^{n - 2} b j_{a,b}'.
    \end{align*}
    Now, if $0 \leq k \leq n - 2$ and $0 \leq \ell \leq n - 2$, then we set \[j_{a,b} = \begin{cases} j_{a,b}' + 1 & k = a \text{ and } \ell = b \\ j_{a,b}' & \text{otherwise} \end{cases}.\] We further extend this by setting $j_{a,b} = 0$ is $a = n-1$ or $b = n-1$, which produces
    \begin{align*}
        \betti[S]{1,1,1}{\frac{S}{J_n}} &= \sum_{a = 0}^{n - 1} \sum_{b = 0}^{n - 1} b j_{a,b} \\
         &= \sum_{a = 0}^{n - 1} \sum_{b = 0}^{n - 1} \left( j_{a,b} - \delta_{b,0} \right) \tbinom{n - 1 - a - b}{0} \tbinom{b}{1}.
    \end{align*}
    If $k = n - 1$ or $\ell = n - 1$, then we set \[j_{a,b} = \begin{cases} 1 & k = a \text{ and } \ell = b \\ j_{a,b}' & a, b < n - 1 \\ 0 & \text{otherwise} \end{cases}.\] The result follows as in the previous case.

    We finally turn our attention to the simpler case of $w > 1$. From the induction hypothesis, we observe that
    \begin{align*}
        \betti[S]{w-1,w-v,v}{\frac{S}{J_{n-1}}} + \betti[S]{w,w+1-v,v}{\frac{S}{J_{n-1}}} \hspace{-5cm} & \\
         =& \sum_{a = 0}^{n - 2} \sum_{b = 0}^{n - 2} \left(j_{a,b}' - \delta_{b,0}\right) \tbinom{n-2-a-b}{w - 1 - v} \tbinom{b}{v} + \sum_{a = 0}^{n - 2} \sum_{b = 0}^{n - 2} \left(j_{a,b}' - \delta_{b,0}\right) \tbinom{n-2-a-b}{w - v} \tbinom{b}{v} \\
         =& \sum_{a = 0}^{n - 2} \sum_{b = 0}^{n - 2} \left(j_{a,b}' - \delta_{b,0}\right) \left[\tbinom{n-2-a-b}{w - 1 - v} + \tbinom{n - 2 - a - b}{w - v} \right] \tbinom{b}{v} \\
         =& \sum_{a = 0}^{n - 2} \sum_{b = 0}^{n - 2} \left(j_{a,b}' - \delta_{b,0}\right) \tbinom{n-1-a-b}{w - v} \tbinom{b}{v}.
    \end{align*}
    Combining this with Theorem \ref{thm:inductMultiBetti} then gives
    \begin{align*}
        \betti[S]{w,w+1-v,v}{\frac{S}{J_n}} &= \tbinom{n - 1 - k - \ell}{w - v} \tbinom{\ell}{v} + \betti[S]{w-1,w-v,v}{\frac{S}{J_{n-1}}} + \betti[S]{w,w+1-v,v}{\frac{S}{J_{n-1}}} \\
         &= \tbinom{n - 1 - k - \ell}{w - v} \tbinom{\ell}{v} + \sum_{a = 0}^{n - 2} \sum_{b = 0}^{n - 2} \left(j_{a,b}' - \delta_{b,0}\right) \tbinom{n-1-a-b}{w - v} \tbinom{b}{v}.
    \end{align*}
    From here, the same reasoning used in the case of $w = 1$ to relate $j_{a,b}$ and $j_{a,b}'$ applies and gives the desired result.
\end{proof}

Similar formulas arise for Betti numbers and graded Betti numbers.

\begin{cor}\label{cor:bettinumbers}
    Given the same setup as in Theorem \ref{thm:xybettis}, the graded Betti numbers of $\CC$ are given by:
    \[
    \betti{w,i}{S/J} = \begin{cases}
        1 & w = i = 0 \\
        \sum_{k = 0}^{n-1}(j_k - 1) \binom{n - 1 - k}{w} & i = w + 1 \geq 2 \\
        0 & \text{otherwise}
    \end{cases}.
    \]
    The total Betti numbers of $\CC$ are given by:
    \[
    \betti{w}{S/J} = \begin{cases}
        1 & w = 0 \\
        \sum_{k = 0}^{n-1}(j_k - 1) \binom{n - 1 - k}{w} & w > 0
    \end{cases}.
    \]
\end{cor}

\begin{proof}
    These formulas can be obtained using the same method as Theorem \ref{thm:xybettis} but using Corollary \ref{cor:gradedBetti} and Corollary \ref{cor:regularBetti}, respectively, in place of Theorem \ref{thm:inductMultiBetti}. Alternatively, these results can be found by combining Theorem \ref{thm:xybettis} with \[\betti[S]{w,i}{S/J} = \sum_{u + v = i} \betti[S]{w,u,v}{S/J}\] and applying combinatorial identities.
\end{proof}

The relationship between the Betti numbers and the piercing numbers can be expressed in a linear equation.

\begin{cor}\label{cor:pierceToBetti}
    Writing $\beta_{w,i} = \betti{w,i}{S/J}$, Corollary~\ref{cor:bettinumbers} can be summarized in the following equation.
    \begin{align*}
   \underset{\vec{\beta}_{+1}(S/J)}{\underbrace{\begin{pmatrix} \beta_{n-1,n} \\ \beta_{n-2,n-1} \\ \vdots \\ \beta_{2,3} \\ \beta_{1,2} \\ \beta_{0,1} \end{pmatrix}}}
 &= 
 \underset{A_n}{\underbrace{\begin{pmatrix} \tbinom{n-1}{n-1} & 0 & \cdots & 0 & 0 & 0 \\ \tbinom{n-1}{n-2} & \tbinom{n-2}{n-2} & & 0 & 0 & 0 \\ \vdots & \vdots & \ddots & & \vdots & \vdots \\ \tbinom{n-1}{2} & \tbinom{n-2}{2} & \cdots & \tbinom{2}{2} & 0 & 0 \\ \tbinom{n-1}{1} & \tbinom{n-2}{1} & \cdots & \tbinom{2}{1} & \tbinom{1}{1} & 0 \\ \tbinom{n-1}{0} & \tbinom{n-2}{0} & \cdots & \tbinom{2}{0} & \tbinom{1}{0} & \binom{0}{0} \end{pmatrix}}}
 \underset{\vec{j}(S/J) - \textbf{1}}{\underbrace{\begin{pmatrix} j_0 - 1 \\ j_1 - 1 \\ \vdots \\ j_{n-3} - 1 \\ j_{n-2} - 1 \\ j_{n - 1} - 1\end{pmatrix}}}
\end{align*}
Moreover, the matrix $A_n$ is an invertible matrix and thus the Betti numbers and piercing numbers are in a one-to-one correspondence. Therefore, the piercing numbers are independent of the piercing order and are given by \[j_k = 1 + \sum_{w = n - 1 - k}^{n - 1} \binom{w}{n - 1 - k} (-1)^{w - n + 1 + k} \betti{w,w+1}{S/J}.\]
\end{cor}

\begin{proof}
    The equation follows from Corollary~\ref{cor:bettinumbers}. The matrix $A_n$ is a lower triangular matrix whose main diagonal consists of only ones, thus $A_n$ is invertible. This puts the Betti numbers and piercing numbers into a one-to-one correspondence. That is, we have
    \[\vec{j}(S/J) - \textbf{1} = A_n^{-1} \cdot \vec{\beta}_{+1}(S/J).\]
    Since the Betti numbers are independent of the piercing order, the piercing numbers must also be independent of the piercing order.

    Moreover, the matrix $A_n$ is sometimes referred to as the Pascal Matrix whose inverse is given by 
    \[A_n^{-1} = \left(\begin{smallmatrix} \tbinom{n-1}{n-1} & 0 & \cdots & 0 & 0 & 0 \\ -\tbinom{n-1}{n-2} & \tbinom{n-2}{n-2} & & 0 & 0 & 0 \\ \vdots & \vdots & \ddots & & \vdots & \vdots \\ (-1)^{n-1}\tbinom{n-1}{2} & (-1)^n\tbinom{n-2}{2} & \cdots & \tbinom{2}{2} & 0 & 0 \\ (-1)^n \tbinom{n-1}{1} & (-1)^{n+1}\tbinom{n-2}{1} & \cdots & -\tbinom{2}{1} & \tbinom{1}{1} & 0 \\ (-1)^{n+1} \tbinom{n-1}{0} & (-1)^{n+2}\tbinom{n-2}{0} & \cdots & \tbinom{2}{0} & -\tbinom{1}{0} & \binom{0}{0} \end{smallmatrix}\right).\] The formula for $j_k$ follows by inputting the above matrix into the linear equation $\vec{j}(S/J) - \textbf{1} = A_n^{-1} \cdot  \vec{\beta}_{+1}(S/J)$.
\end{proof}

A similar result holds for $j_{k,\ell}$, however the linear algebra approach is not quite as clean as the above proof. Instead, we obtain the following result using combinatorial formulas.

\begin{thm}\label{thm:jkl}
Let $\CC$ be an inductively pierced neural code on $n$ neurons, and $J$ the \pci of $\CC$ in the ring $S$. Given a piercing order, let $j_{a,b}$ denote the number of $a$-piercings of $\CC$ contained in exactly $b$ other place fields. Then
    \[j_{a,b} = \delta_{b,0} + \sum_{v = 0}^{n - 1} \sum_{w = 0}^{n - 1} \betti{w,w+1-v,v}{S/J} \binom{w - v}{n - 1 - a - b} \binom{v}{b} (-1)^{w - n + 1 + a}.\]
    As a consequence, the $j_{a,b}$ are independent of the choice of piercing order.
\end{thm}

\begin{rem}
    This corollary implies that the results of this section do not depend on the choice of piercing order.
\end{rem}

\begin{proof}
For brevity, set $\alpha_{k,\ell} = j_{k, \ell} - \delta_{\ell,0}$.
\begin{align*}
    \sum_{v = 0}^{n - 1} \sum_{w = 0}^{n - 1} \betti[S]{w,w+1-v,v}{\frac{S}{J}} \tbinom{w - v}{n - 1 - a - b} \tbinom{v}{b} (-1)^{w - n + 1 + a} \hspace{-3.5cm} & \\
    =& \sum_{v = 0}^{n - 1} \sum_{u = -v}^{n - 1 - v} \betti[S]{u+v,u+1,v}{\frac{S}{J}} \tbinom{u}{n - 1 - a - b} \tbinom{v}{b} (-1)^{u + v - n + 1 + a}
\end{align*}
We note that if $v > 0$, then whenever $-v \leq u < 0$, we have $\tbinom{u}{n - 1 - a - b} = 0$. Thus, we can change the lower bound of the inner sum from $u = -v$ to $u = 0$ to get
\[ \sum_{v = 0}^{n - 1} \sum_{u = 0}^{n - 1 - v} \betti[S]{u+v,u+1,v}{\frac{S}{J}} \tbinom{u}{n - 1 - a - b} \tbinom{v}{b} (-1)^{u + v - n + 1 + a}. \] Next, if $u > n - 1 - v$, then $w = u + v > n - 1$ which implies $\betti[S]{w,u,v}{S/J} = 0$ since the projective dimension of $S/J$ is, at most, $n - 1$. As such the upper bound of the inner sum can be changed from $u = n - 1 - v$ to $u = n - 1$. With this we obtain the following.
\begin{align*}
    \sum_{v = 0}^{n - 1} \sum_{w = 0}^{n - 1} \betti[S]{w,w+1-v,v}{\frac{S}{J}} \tbinom{w - v}{n - 1 - a - b} \tbinom{v}{b} (-1)^{w - n + 1 + a} \hspace{-5cm} & \\
    =& \sum_{v = 0}^{n - 1} \sum_{u = 0}^{n - 1} \betti[S]{u+v,u+1,v}{\frac{S}{J}} \tbinom{u}{n - 1 - a - b} \tbinom{v}{b} (-1)^{u + v - n + 1 + a} \\
    =& \sum_{v = 0}^{n - 1} \sum_{u = 0}^{n - 1} \sum_{k = 0}^{n - 1} \sum_{\ell = 0}^{n - 1} \alpha_{k,\ell} \tbinom{n - 1 - k - \ell}{u} \tbinom{\ell}{v} \tbinom{u}{n - 1 - a - b} \tbinom{v}{b} (-1)^{u + v - n + 1 + a} \\
    =& \sum_{k = 0}^{n - 1} \sum_{\ell = 0}^{n - 1} \alpha_{k,\ell} \sum_{v = 0}^{n - 1} \sum_{u = 0}^{n - 1} \tbinom{\ell}{v} \tbinom{v}{b} \tbinom{n - 1 - k - \ell}{u}  \tbinom{u}{n - 1 - a - b}  (-1)^{u + v - n + 1 + a}
\end{align*}
Next, we use Example \ref{ex:comb}\eqref{comb:prod} to get $\tbinom{\ell}{v} \tbinom{v}{b} = \tbinom{\ell}{b} \tbinom{\ell - b}{v - b}$. With this equality, we find
\begin{align*}
    \sum_{v = 0}^{n - 1} \sum_{w = 0}^{n - 1} \betti[S]{w,w+1-v,v}{\frac{S}{J}} \tbinom{w - v}{n - 1 - a - b} \tbinom{v}{b} (-1)^{w - n + 1 + a} \hspace{-6cm} & \\
    =& \sum_{k = 0}^{n - 1} \sum_{\ell = 0}^{n - 1} \alpha_{k,\ell} \sum_{v = 0}^{n - 1} \sum_{u = 0}^{n - 1} \tbinom{\ell}{b} \tbinom{\ell - b}{v - b} \tbinom{n - 1 - k - \ell}{u}  \tbinom{u}{n - 1 - a - b}  (-1)^{u + v - n + 1 + a} \\
    =& \sum_{k = 0}^{n - 1} \sum_{\ell = 0}^{n - 1} \alpha_{k,\ell} \sum_{u = 0}^{n - 1} \sum_{v = 0}^{n - 1} \tbinom{\ell}{b} \tbinom{\ell - b}{v - b} \tbinom{n - 1 - k - \ell}{u}  \tbinom{u}{n - 1 - a - b}  (-1)^{u + v - n + 1 + a} \\
    =& \sum_{k = 0}^{n - 1} \sum_{\ell = 0}^{n - 1} \alpha_{k,\ell} \tbinom{\ell}{b} \sum_{u = 0}^{n - 1} \tbinom{n - 1 - k - \ell}{u}  \tbinom{u}{n - 1 - a - b}  (-1)^{u + b - n + 1 + a} \sum_{v = 0}^{n - 1} \tbinom{\ell - b}{v - b} (-1)^{v - b} \\
    =& \sum_{k = 0}^{n - 1} \alpha_{k,b} \tbinom{b}{b} \sum_{u = 0}^{n - 1} \tbinom{n - 1 - k - b}{u}  \tbinom{u}{n - 1 - a - b}  (-1)^{u + b - n + 1 + a}
\end{align*}
where the last equality follows from Example \ref{ex:comb}\eqref{comb:AltSum} gives \[\sum_{v = 0}^{n - 1} \tbinom{\ell - b}{v - b} (-1)^{v - b} = \sum_{v = b}^{\ell} \tbinom{\ell - b}{v - b} (-1)^{v - b} = \sum_{v = 0}^{\ell - b} \tbinom{\ell - b}{v} (-1)^v = \begin{cases} 1 & \ell = b \\ 0 & \ell \neq b \end{cases}.\]

By similar reasoning, we get $\tbinom{n - 1 - k - b}{u} \tbinom{u}{n - 1 - a - b} = \tbinom{n - 1 - k - b}{n - 1 - a - b} \tbinom{a - k}{u - n + 1 + a + b}$ and \[\sum_{u = 0}^{n - 1} \tbinom{a - k}{u - n + 1 + a + b} (-1)^{u - n + 1 + a + b} = \begin{cases} 1 & k = a \\ 0 & k \neq a \end{cases}.\] This allows for the following computation.
\begin{align*}
    \sum_{v = 0}^{n - 1} \sum_{w = 0}^{n - 1} \betti[S]{w,w+1-v,v}{\frac{S}{J}} \tbinom{w - v}{n - 1 - a - b} \tbinom{v}{b} (-1)^{w - n + 1 + a} \hspace{-3.5cm} & \\
    =& \sum_{k = 0}^{n - 1} \alpha_{k,b} \tbinom{b}{b} \sum_{u = 0}^{n - 1} \tbinom{n - 1 - k - b}{n - 1 - a - b} \tbinom{a - k}{u - n + 1 + a + b}  (-1)^{u + b - n + 1 + a} \\
    =& \sum_{k = 0}^{n - 1} \alpha_{k,b} \tbinom{b}{b} \tbinom{n - 1 - k - b}{n - 1 - a - b}  \sum_{u = 0}^{n - 1} \tbinom{a - k}{u - n + 1 + a + b}  (-1)^{u + b - n + 1 + a} \\
    =& \alpha_{a,b} \tbinom{b}{b} \tbinom{n - 1 - a - b}{n - 1 - a - b} \\
    =& \alpha_{a,b} \\
    =& j_{a,b} - \delta_{b,0}.
\end{align*}
Adding $\delta_{b,0}$ to both sides provides the result.
\end{proof}

Next we compute the projective dimensions of inductively pierced codes.

\begin{cor}\label{cor:pdim}
    Let $\CC$ be a $k$-inductively pierced neural code on $n$ neurons, and $J$ the \pci of $\CC$ in the ring $S$.

    Given a piercing order, if $j_0 = \cdots = j_{t-1} = 1$ for some $t > 0$, then $\betti{n-u-1}{S/J} = 0$ for $0 \leq u < t$.

    Moreover, if $t \ge 0$ is the smallest integer such that $j_t >1$, 
    then $\pdim_S S/J = n - 1 - t$. 
\end{cor}

\begin{proof}
    We start by assuming $j_0 = \cdots = j_{t-1} = 1$ for some $t > 0$ and let $0 \leq u < t$. We then find the following:
    \begin{align*}
        \betti{n-u-1}{S/J} &= \sum_{k = 0}^{n - 1} \left(j_k - 1\right) \binom{n - 1 - k}{n - 1 - u} \\
         &= \sum_{k = t}^{n - 1} \left(j_k - 1\right) \binom{n - 1 - k}{n-1-u}.
    \end{align*}
    Thus we have $u < t \leq k$, and so $n - 1 - u > n - 1 - k$. This yields $\tbinom{n-1-k}{n-1-u} = 0$ for all $0 \leq u < t$ and thus $\betti{n-u-1}{S/J} = 0$.

    If we further suppose $j_t > 1$, then we obtain \[\betti{n-t-1}{S/J} = \sum_{k = t}^{n - 1} \left(j_k - 1\right) \binom{n - 1 - k}{n-1-t} = \left(j_t - 1\right) \binom{n - 1 - t}{n-1-t} = j_t - 1 > 0. \] Since $\betti[S]{n-1-t}{S/J} \neq 0$ but $\betti[S]{n-u-1}{S/J} = 0$ for $0 \leq u < t$, we must have $\pdim_S S/J = n - 1 - t$.
\end{proof}

\begin{rem}
    While the multigraded Betti numbers of an inductively pierced neural code do give substantial information on the piercings, they can't completely distinguish all features of codes. For example, the following two neural ideals on 5 neurons have the same multigraded Betti numbers:
    \begin{align*}
        &I_A = (x_1x_3,x_1x_4,x_1x_5,x_2x_4,x_2x_5), \\
        &I_B = (x_1x_3,x_1x_4,x_1x_5,x_2x_4,x_4x_5).
    \end{align*}

    \begin{figure}[h]
\includegraphics[width=2in]{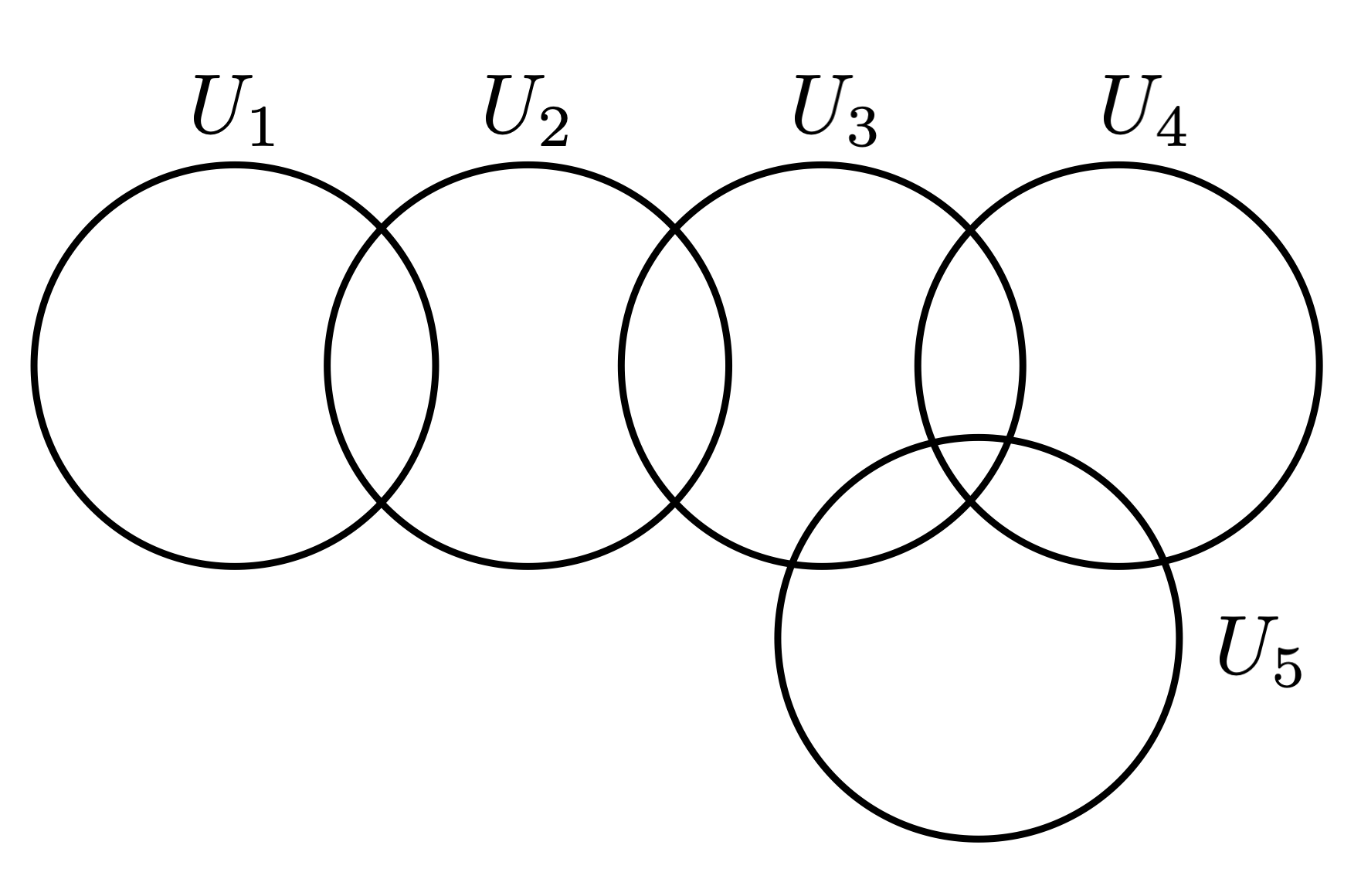} \hspace{0.5in} \includegraphics[width=2in]{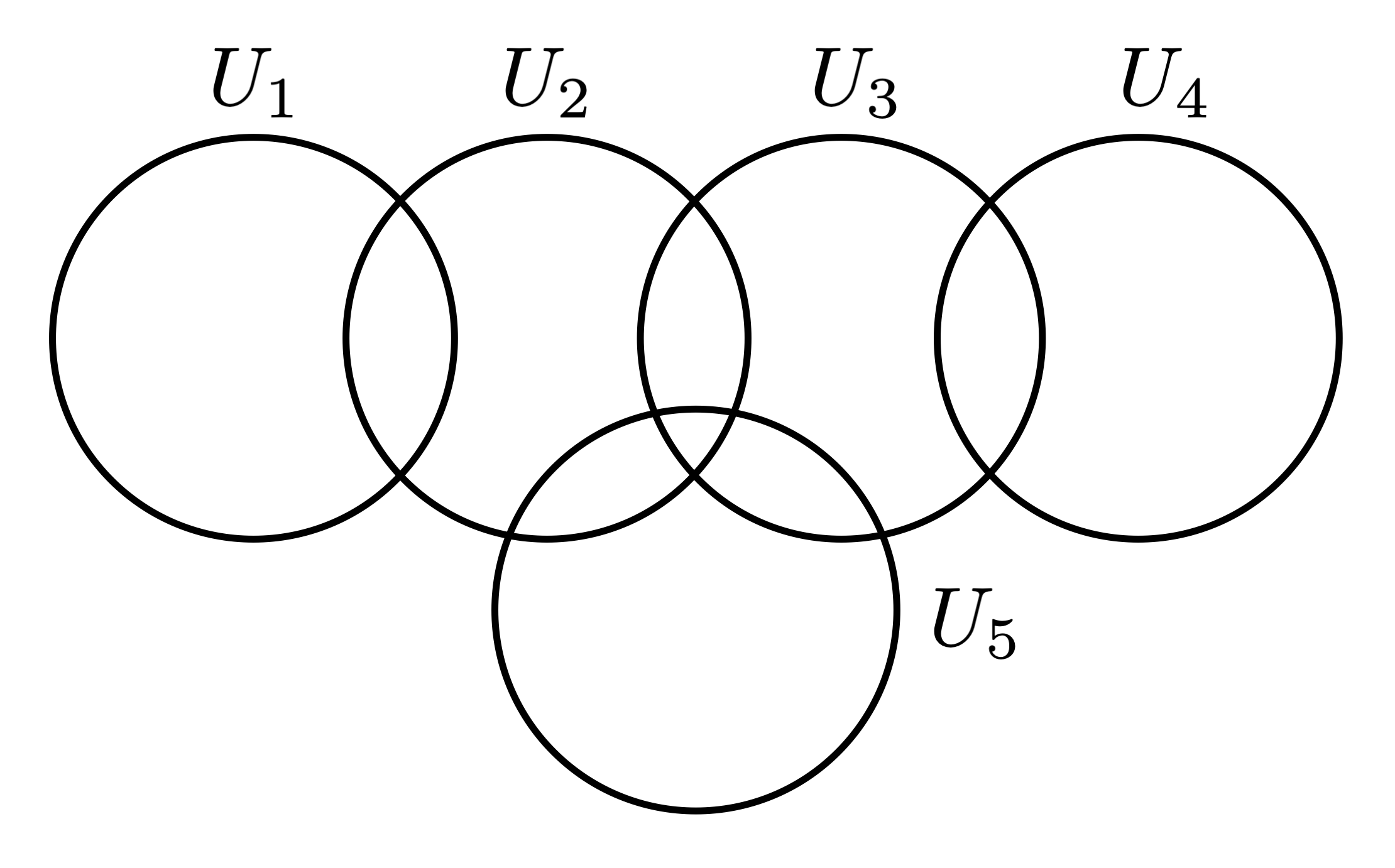}
\caption{At left, a realization of the code of ideal $I_A$. At right, that of $I_B$. }\label{fig:grBettiToric}
\end{figure}
    Figure~\ref{fig:grBettiToric} shows realizations of these two codes.
    However, a Macaulay2 computation \cite{M2} shows that the toric ideals of the corresponding neural codes (as defined in \cite{GrossObatakeYoungs}) are distinct; in particular the toric ideal of the first example has 15 generators, whereas the toric ideal of the second example has 17 generators. This does hint that the toric ideal of a neural code is a finer invariant than even multigraded Betti numbers. However, current results on toric ideals of neural codes only describe connections to $k$-inductively pierced codes for small values of $k$ \cites{toricideals,toricideals2}.
\end{rem}

\section{A result on chordal graphs}\label{sec:chordal}

Inductively pierced codes, as previously noted, are closely related to chordal graphs. In this section, we use the connection between the two to show how one of our own results (namely, Corollary~\ref{cor:pierceToBetti}) also proves a theorem about chordal graphs. First, we provide some definitions, and a well-known result, all as in \cite{West}.

\begin{defn} A graph $G$ is {\it chordal} if it has no chordless cycle (i.e., every cycle of length $\geq 4$ has a chord). 
\end{defn}

One notable property of chordal graphs is that they can be deconstructed by removing so-called simplicial vertices. In fact, this feature characterizes the graph.

\begin{defn} A vertex of a graph $G=(V,E)$ is {\it simplicial} if its neighborhood forms a clique. That is, $v\in V$ is simplicial if for all vertices $u,w$ so that $\{v,w\}$ and $\{v,u\}$ form edges, $\{u,w\}$ also forms an edge. 

\end{defn}

\begin{defn}
A {\it simplicial elimination ordering} is an ordering $v_n,...,v_1$ for deletion of vertices so that each vertex $v_{i+1}$ is a simplicial vertex of the remaining graph induced by $v_1,...,v_i$.
\end{defn}

\begin{thm}

A graph is chordal if and only if it has a simplicial elimination ordering.

\end{thm}

In previous work \cite{curry2022recognizing} it has been shown that inductively pierced codes are characterized, in part, by their similarity to chordal graphs. The following definition and theorem are from \cite{curry2022recognizing}, but the theorem is provided here (as in Section~\ref{sec:defn}) in its polarized form.

{\renewcommand{\thethm}{\ref{def:grc}}
\begin{defn} Let $\CC$ be a code on $n$ neurons that has a degree two canonical form. The \textit{general relationship graph} $G(\CC)$ is the graph with vertices $V = [n]$, and an edge ${i,j}$ if and only if the canonical form of $J$ does {\it not} contain any of the pseudo-monomials $x_ix_j$, $x_iy_j$, or $x_jy_i$.
\end{defn}
\addtocounter{thm}{-1}
}

{
\renewcommand{\thethm}{\ref{thm:IpChar}}
\begin{thm}
A neural code $\CC$ is inductively pierced if and only if
\begin{enumerate}
\item $CF(\CC)$ consists entirely of degree two polynomials  and 
\item the graph $G(\CC)$ is chordal.
\end{enumerate}
\end{thm}
\addtocounter{thm}{-1}
}

Intuitively, this result holds because the two structures have a major similarity: both can be built up (or taken apart, depending on your point of view) by successively adding (removing) an object that relates to its neighbors at that stage in a very particular way. A piercing order is usually viewed as constructive and so is the reverse of an elimination order, but the concept is the same. 

The proof of this characterization (in \cite{ curry2022recognizing}) shows that a simplicial elimination ordering on the graph can always be found that reverses to a piercing order; however, all the difficulties of the proof come from the existence of pseudo-monomials of the form $x_iy_j$, and ensuring that the order is consistent with these.  In a code with no $y_j$ appearing, {\it any} simplicial elimination ordering on $G(\CC)$ reverses to a piercing order.  Moreover, the degree $k$ of the simplicial vertex in the elimination ordering corresponds precisely to the $k$ of the $k$-piercing in the corresponding piercing order. 

We also note that every chordal graph is the general relationship graph for a neural code (in fact, usually for many!). Thus, we can use our Corollary~\ref{cor:pierceToBetti},  which implies that the number of $k$-piercings in any piercing order is a characteristic of the code, to obtain the following result on chordal graphs.

\begin{thm}
\label{thm:simplicialfixed}
Let $G$ be a chordal graph. The number of vertices of simplicial degree $k$ removed in any simplicial elimination ordering is fixed.
\end{thm}

\section{Acknowledgments}

The authors would like to thank Michael DeBellevue, Carina Curto, Juliann Geraci, Selvi Kara, and Keri Ann Sather-Wagstaff for helpful conversations throughout the process of writing this paper. The first two authors would also like to thank Justin Roy Cox, Gabriel Lumpkin, Swan Klein, and Connor William Poulton for their work on a semester-long project through the Mason Experimental Geometry Lab, in which they computed the Betti numbers of several families of neural codes. Their work ultimately convinced the authors to look for different classes of neural codes to study, resulting in the research in this paper. During the writing of this paper, the second-named author was partially supported by NSF grant DMS-2424326 and the third-named author by NSF grant DMS–2401482. The last-named author was supported by the Haynesville Project.

This work represents the views of the authors and is not to be regarded as representing the opinions
of the Center for Naval Analyses or any of its sponsors.

\appendix

\section{Additional homological background material}
\label{sec:betti}

In this section we define some of the algebraic terms used in the paper, and give references for the reader to learn more about them.

Much of the material immediately following can be found in \cite{Eisenbud}.

\begin{defn}
    Let $S$ be a commutative Noetherian local (resp. graded) ring, $k=S/m$ where $m$ is the (homogeneous) maximal ideal, and $M$ a \fg\ (graded) $S$-module. A free resolution $F_\bullet$ of $M$ is \textit{minimal} if the differentials of $F_\bullet \otimes k$ are all equal to 0.
\end{defn}

\begin{defn}
\label{def:betti}
    Let $S$ be a commutative Noetherian graded ring and $M$ a \fg\ graded $S$-module. Take $F_\bullet$ to be a minimal graded free resolution of $M$. 
    
    \begin{enumerate}
        \item The \textit{Betti numbers} of $M$ are given by
    \[\beta_i(M):=\rank(F_i).\]
        \item The \textit{graded Betti numbers} of $M$ are given by
    \[\beta_{i,j}(M):=\rank(F_i(-j)).\]
        \item The \textit{regularity} of $M$ is the smallest $j$ such that for every $i$, $F_i$ is generated in degree less than or equal to $i+j$. 
        \item For an ideal $I$ of $R$, we may compute the regularity of $I$ as an $R$-module, or of $R/I$. The former is computed by resolving $I$ as an $R$-module, and the latter by resolving $R/I$ as an $R$-module. Given a minimal resolution of $R/I$, one can construct a minimal resolution of $I$ by removing the copy of $R$ in homological degree 0. In many of the examples in this paper, $I$ is generated in degree 2, so $\reg{I}=\reg{R/I}-1$.
    \end{enumerate}
    
\end{defn}

Below are some results on Betti numbers and Poincar\'{e} series that follow immediately from their definitions: 

\begin{rem} \label{rem:bettiformulas}
Let $S$ be a polynomial ring on at least $n$ variables, $\fp \subset S$ a monomial prime generated by $n - 1 - \ell$ variables, and $x \in S$. It is a common result that $S / \fp$ is minimally resolved by the Koszul complex on the generators of $\fp$. As such, one finds that
\[\betti[S]{i}{\frac{S}{\fp}} = \binom{n - 1 - \ell}{i}.\]
Similarly, $S / (x)$ is also resolved by a Koszul complex. Specifically, it is resolved the by Koszul complex on $x$ which yields
\[\betti[S]{i}{\frac{S}{(x)}} = \begin{cases} 1 & i \in \{0,1\} \\ 0 & i \notin \{0,1\}\end{cases}.\]
Finally, $S / (0) = S$ is resolved by the ring $S$ itself viewed as a complex concentrated in homological degree 0. It follows that
\[\betti[S]{i}{\frac{S}{(0)}} = \begin{cases} 1 & i = 0 \\ 0 & i \neq 0 \end{cases}.\]
\end{rem}

\begin{defn}
    Suppose $F_{\bullet}$ and $G_{\bullet}$ are complexes of $S$-modules with differentials $\partial_i^F$ and $\partial_i^G$, respectively. A {\it chain map} $\phi: F_{\bullet} \to G_{\bullet}$ is a collection of $S$-module homomorphisms $\phi_i :  F_i \to G_i$ such that \[\phi_{i-1} \circ \partial_{i}^F = \partial_i^G \circ \phi_i \] for all $i \in \Z$, i.e. homomorphisms that form commutative squares with the complex maps on $F_\bullet$ and $G_\bullet$.
\end{defn}

We will also need the following definition and result:

\begin{defn}[\cite{CommRingTheory}*{Section 16, page 123}]\label{def:regSeq}
    Let $R$ be a ring and $M$ an $R$-module. An element $a \in R$ is said to be \emph{$M$-regular} if $ax \neq 0$ for all $0 \neq x \in M$. A sequence $a_1,\ldots, a_p \in R$ is an \emph{$M$-sequence} (or an \emph{$M$-regular sequence}) if the following three conditions hold:
    \begin{enumerate}
        \item $a_1$ is $M$-regular, $a_2$ is $M / a_1 M$-regular,
        \item $a_i$ is $M / (a_1M + \cdots + a_{i - 1}M)$-regular for $2 \leq i \leq p$,
        \item $M / (a_1 M + \cdots a_p M) \neq 0$.
    \end{enumerate}
\end{defn}

For more information about Tor, see for example \cite{Eisenbud}. The following results contain all of the necessary information about Tor for this paper.

\begin{thm}[\cite{CommRingTheory}*{Theorem 16.5(i)}]\label{thm:CommRing}
    Let $R$ be a ring, $M$ an $R$-module, and $a_1,\ldots,a_p$ an $M$-sequence, then $\Tor{R}{R/(a_1,\ldots,a_p)}{M} = 0$ for $i > 0$.
\end{thm}

\begin{cor}\label{cor:freeresregseq}
    Let $R$ be a ring, $M$ an $R$-module, $a_1,\ldots,a_p$ an $M$-sequence, and $Q=R/(a_1,\ldots,a_p)$. If $F_\bullet$ is a free resolution of $M$ over $R$, then $F_\bullet \otimes Q$ is a free resolution of $M$ over $Q$.

    In particular, the number of copies of $Q$ in homological degree $i$ is equal to the number of copies of $R$ in homological degree $i$, and the matrices of the differentials are given by replacing the elements of $R$ with their images modulo $a_1,\ldots,a_q$. This implies that if $R$ is local (resp. graded) and $F_\bullet$ is a minimal (resp. graded) resolution, then so is $F_\bullet \otimes Q$. 
\end{cor}

\begin{thm}[Balance of Tor, {\cite{weibel}*{proof of Theorem 2.7.2}}]\label{thm:BalanceTor}
    Let $M$ and $N$ be $R$-modules resolved over $R$ by free resolutions $\cX$ and $\cY$, respectively. Then, for all $i \geq 0$, the $R$-module $\Tor[i]{R}{M}{N}$ is equal to \[H_i(\cX \otimes_R N) \cong H_i(\cX \otimes_R \cY) \cong H_i(M \otimes_R \cY). \]
\end{thm}

\begin{cor}\label{cor:tensorRes}
    Consider the same conditions as Theorem~\ref{thm:BalanceTor}. If \[\Tor[i]{R}{M}{N} = 0\] for all $i > 0$, then the $R$-complex $\cX \otimes_R \cY$ resolves $M \otimes_R N$. Moreover, if $R$ is local (resp. graded) and $\cX$ and $\cY$ are both (resp. graded) minimal resolutions, then so is $\cX \otimes_R \cY$.
\end{cor}

The following are standard results, whose proofs we include for the sake of keeping the paper self-contained.

\begin{cor}\label{cor:tensorBetti}
    Consider the same conditions as Corollary~\ref{cor:tensorRes}, and in particular assume $R$ is local or graded. The Betti numbers of $M \otimes_R N$ over $R$ satisfy \[\betti[R]{w}{M \otimes_R N} = \sum_{i = 0}^w \betti[R]{i}{M} \betti[R]{w - i}{N}\] for all $w \geq 0$. Similarly, if $R$ is graded and $M$ and $N$ are both graded $R$-modules, then \[\betti[R]{w, u}{M \otimes_R N} = \sum_{i = 0}^w \sum_{j = 0}^{u} \betti[R]{i, j}{M} \betti[R]{w - i, u - j}{N}\] for all $w , u \geq 0$. Lastly, if $R$ has a multigrading given by $\deg(x)=a$, $\deg(y)=b$, and $M$ and $N$ are multigraded, then \[\betti[R]{w, u, v}{M \otimes_R N} = \sum_{i = 0}^w \sum_{j = 0}^{u} \sum_{\ell = 0}^v \betti[R]{i, j, \ell}{M} \betti[R]{w - i, u - j, v - \ell}{N}\] for all $w, u, v \geq 0$.
\end{cor}

\begin{proof}
    Suppose $\cX$, $\cY$, and $\mathcal{Q}$ are the minimal free resolutions of $M$, $N$, and $M \otimes_R N$, respecitvely.  By Corollary~\ref{cor:tensorRes}, we have \[\mathcal{Q}_w \cong \bigoplus_{i = 0}^w \left(\cX_i \otimes_R \cY_{w - i}\right).\] It follows that
    \begin{align*}
        R^{\betti[R]{w}{M \otimes_R N}} &\cong \mathcal{Q}_w \\
         &\cong \bigoplus_{i = 0}^w \left(\cX_i \otimes_R \cY_{w - i}\right) \\
         &\cong \bigoplus_{i = 0}^w \left(R^{\betti[R]{i}{M}} \otimes_R R^{\betti[R]{w - i}{N}} \right) \\
         &\cong \bigoplus_{i = 0}^w R^{\betti[R]{i}{M} \betti[R]{w - i}{N}} \\
         &\cong R^{\sum_{i = 0}^w \betti[R]{i}{M} \betti[R]{w - i}{N}}.
    \end{align*}
    Thus, we have \[\betti[R]{w}{M \otimes_R N} = \sum_{i = 0}^w \betti[R]{i}{M} \betti[R]{w - i}{N}.\] An analogous argument holds for the graded and multigraded Betti numbers of $M \otimes_R N$ over $R$.
\end{proof}

\section{Combinatorial Identities}
\label{sec:combo}

We list the combinatorial identities used in this paper. All of these are commonly used, so we omit their proofs.

\begin{rem}
    Given $a,b \in \Z$, we adopt the usual definition of the binomial coefficient $\binom{a}{b}$. That is, the binomial coefficient is always given by \[\binom{a}{b} = \begin{cases}  \frac{a!}{b!(a-b)!} & 0 \leq b \leq a \\ 0 & \text{otherwise}. \end{cases}\]
\end{rem}

\begin{example}
\label{ex:comb}
All variables represent integers unless otherwise specified. 
    \begin{enumerate}
    \item \label{comb:symm} $\binom{a}{b} = \binom{a}{a-b}$.
    \item \label{comb:Pascal} (Pascal's Rule)  \[\binom{a}{b} + \binom{a}{b+1} = \binom{a+1}{b+1}.\] 
    \item \label{comb:HS} (Hockey Stick Identity) \[\sum_{m = b}^a \binom{m}{b} = \binom{a+1}{b+1}.\]
    \item \label{comb:AltSum} \[\sum_{m=0}^a (-1)^m \binom{a}{m} = \begin{cases} 1 & a = 0 \\ 0 & a \neq 0 \end{cases}.\]
    \item \label{comb:ChuV} (Chu-Vandermonde Identity) \[\sum_{m=0}^c \binom{a}{m} \binom{b}{c-m} = \binom{a+b}{c}.\]
    \item \label{comb:prod} \[\binom{a}{b} \binom{b}{c} = \binom{a}{c} \binom{a-c}{b-c} = \binom{a}{c} \binom{a-c}{a-b}.\]
\end{enumerate}
\end{example}

\bibliographystyle{amsalpha}
\bibliography{neuralbib}
\end{document}